\documentclass[10pt,reqno]{amsart}

\usepackage[a4paper,headheight=0.3cm]{geometry}
\usepackage{fancyhdr}
\usepackage{url}
\usepackage[pdftitle={On the connection between probability boxes and possibility measures},%
pdfauthor={Matthias C. M. Troffaes, Enrique Miranda, and Sebastien Destercke},%
pdfkeywords={probability boxes, possibility measures, maxitive measures, coherent lower and upper probabilities, natural extension}]{hyperref}

\usepackage{mathabx} 
\usepackage{mathptmx}
\usepackage{amssymb}
\usepackage{mathtools,amsthm}
\usepackage{enumerate}
\usepackage{hyperref}
\usepackage{bm}
\usepackage{tikz}
\usetikzlibrary{snakes,arrows}
\usepackage{comment}

\newcommand{\df}{F}

\newcommand{\ldf}{\underline{\df}}
\newcommand{\udf}{\overline{\df}}

\newcommand{\pdomain}{\mathcal{K}}
\newcommand{\plattice}{\mathcal{H}}
\newcommand{\pr}{P}

\newcommand{\nex}{E}
\newcommand{\lpr}{\underline{\pr}}
\newcommand{\upr}{\overline{\pr}}

\newcommand{\lnex}{\underline{\nex}}
\newcommand{\unex}{\overline{\nex}}

\newcommand{\upbox}{\upr_{\ldf,\udf}}

\newcommand{\upboxlattice}{{\upr^{\plattice}_{\ldf,\udf}}}

\newcommand{\lnepbox}{\lnex_{\ldf,\udf}}
\newcommand{\unepbox}{\unex_{\ldf,\udf}}
\newcommand{\lnepossib}{\lnex_{\Pi}}
\newcommand{\lcdf}{\lpr_\df}
\newcommand{\lnecdf}{\lnex_\df}

\newcommand{\set}[2]{\left\{#1\colon#2\right\}}
\newcommand{\gambles}{\mathcal{L}}
\newcommand{\solp}{\mathcal{M}}

\newcommand{\pspace}{\Omega}
\newcommand{\allprobs}{\mathcal{P}} 

\newcommand{\reals}{\mathbb{R}}
\newcommand{\values}{\mathcal X}

\theoremstyle{plain}
\newtheorem{theorem}{Theorem}
\newtheorem{proposition}[theorem]{Proposition}
\newtheorem{corollary}[theorem]{Corollary}
\newtheorem{lemma}[theorem]{Lemma}
\theoremstyle{remark}
\newtheorem{example}[theorem]{Example}

\theoremstyle{definition}
\newtheorem{definition}[theorem]{Definition}

\begin{document}
\title{On the connection between probability boxes and possibility measures}
\author[M. Troffaes]{Matthias C. M. Troffaes}
\address{Durham University, Dept. of Mathematical Sciences, Science Laboratories, South Road, Durham
DH1 3LE, United Kingdom} \email{matthias.troffaes@gmail.com}
\author[E. Miranda]{Enrique Miranda}
\address{University of Oviedo, Dept. of Statistics and O.R. C-Calvo Sotelo, s/n, Oviedo, Spain}
\email{mirandaenrique@uniovi.es}
\author[S. Destercke]{Sebastien Destercke}
\address{CIRAD, UMR1208, 2 place P. Viala, F-34060 Montpellier cedex 1, France} \email{desterck@irit.fr}

\begin{abstract}
  We explore the relationship between possibility measures (supremum preserving normed measures) and p-boxes (pairs of cumulative distribution functions) on totally preordered spaces, extending earlier work in this direction by
  De Cooman and Aeyels, among others.
  We start by demonstrating that only
  those p-boxes who have $0$--$1$-valued lower or upper cumulative
  distribution function can be possibility measures, and we derive
  expressions for their natural extension in this case. Next, we
  establish necessary and sufficient conditions for a p-box to be a
  possibility measure. Finally, we show that almost every possibility
  measure can be modelled by a p-box. Whence, any techniques for
  p-boxes can be readily applied to possibility measures. We
  demonstrate this by deriving joint possibility measures from
  marginals, under varying assumptions of independence, using a
  technique known for p-boxes.
  Doing so, we arrive at a new rule of
  combination for possibility measures, for the independent case.
\end{abstract}


\maketitle
\thispagestyle{fancy}

\section{Introduction}

Firstly, possibility measures are supremum preserving set functions,
and were introduced in fuzzy set theory \cite{zadeh1978a},
although earlier appearances exist \cite{shackle1961,lewis1973a}.
Because of their computational simplicity,
possibility measures are widely applied in many fields,
including data analysis
\cite{tanaka1999}, diagnosis \cite{cayrac1996}, cased-based
reasoning \cite{hullermeier2007}, and psychology
\cite{raufaste2003}.
This paper concerns quantitative possibility theory
\cite{dubois1988}, where degrees of
possibility range in the unit interval. Interpretations
abound
\cite{dubois2006}: we can see them as likelihood
functions \cite{dubois1997b}, as particular cases of plausibility
measures \cite{shafer1976,shafer1987}, as extreme probability
distributions \cite{Spohn88}, or as upper
probabilities \cite{walley1996,cooman1998}. The upper probability
interpretation fits our purpose best, whence is assumed herein.

Secondly, probability boxes
\cite{ferson2003,2006:ferson}, or p-boxes for short, are
pairs of lower and upper cumulative distribution functions, and are often used
in risk and safety studies, in which they play an
essential role.
P-boxes have been connected to info-gap theory \cite{2008:ferson},
random sets \cite{KrieglerHeld2005,2008:oberguggenberger}, and
also, partly, to possibility measures \cite{2006:baudrit,cooman1998}. P-boxes can
be defined on arbitrary finite spaces \cite{desterckedubois2008}, and,
more generally, even on arbitrarily totally pre-ordered spaces
\cite{unpub:troffaesdestercke::pboxes:multivariate}---we will use this
extensively.

This paper aims to consolidate the connection between
possibility measures and p-boxes, making as few assumptions as possible.
We prove that almost every possibility measure can
be interpreted as a p-box, whence, p-boxes effectively generalize
possibility measures. Conversely, we provide necessary
and sufficient conditions for a p-box to be a possibility measure,
whence, providing conditions under which the more efficient mathematical
machinery of possibility measures is applicable to p-boxes.

To study this connection, we use imprecise probabilities \cite{walley1991},
because both possibility
measures and p-boxes are particular cases of imprecise probabilities.
Possibility measures are explored as imprecise probabilities in \cite{walley1996,cooman1998,miranda2003}, and p-boxes were studied as
imprecise probabilities briefly in
\cite[Section~4.6.6]{walley1991} and \cite{troffaes2005}, and in
much more detail in
\cite{unpub:troffaesdestercke::pboxes:multivariate}.

The paper is organised as follows: in Section~\ref{sec:prel}, we
give the basics of the behavioural theory of imprecise
probabilities, and recall some facts about p-boxes and possibility
measures; in Section~\ref{sec:pbox-as-max}, we first determine
necessary and sufficient conditions for a p-box to be maximum
preserving, before determining in Section~\ref{sec:pbox-to-poss}
necessary and sufficient conditions for a p-box to be a possibility
measure; in Section~\ref{sec:poss-to-pbox}, we show that almost any
possibility measure can be seen as particular p-box, and that many
p-boxes can be seen as a couple of possibility measures; some
special cases are detailed in Section~\ref{sec:specialcases}.
Finally, in Section~\ref{sec:marginals} we apply the work on
multivariate p-boxes from
\cite{unpub:troffaesdestercke::pboxes:multivariate} to derive
multivariate possibility measures from given marginals, and in
Section~\ref{sec:conclusions} we give a number of additional
comments and remarks.

\section{Preliminaries}
\label{sec:prel}

\subsection{Imprecise Probabilities}
\label{sec:imprecise:probabilities}

We start with a brief introduction to imprecise probabilities (see
\cite{1854:boole,williams1975,walley1991,miranda2008} for more
details). Because possibility measures are interpretable as upper
probabilities, we start out with those, instead of lower
probabilities---the resulting theory is equivalent.

Let $\pspace$ be the possibility space. A subset of $\pspace$ is
called an \emph{event}. Denote the set of all events by
$\wp(\pspace)$, and the set of all finitely additive
probabilities on $\wp(\pspace)$ by $\allprobs$.

In this paper, an \emph{upper probability} is any real-valued function $\upr$ defined on an arbitrary subset
$\pdomain$ of $\wp(\pspace)$.
With $\upr$, we associate a
\emph{lower probability} $\lpr$ on $\{A\colon A^c \in
\pdomain\}$ via the conjugacy relationship $$\lpr(A)=1-\upr(A^c).$$

Denote the set of all finitely additive probabilities on $\wp(\pspace)$ that are dominated by $\upr$ by:
\begin{equation*}
  \solp(\upr)=\{\pr\in\allprobs\colon (\forall A\in\pdomain)(\pr(A)\le\upr(A))\}
\end{equation*}
Clearly, $\solp(\upr)$ is also the set of all finitely additive
probabilities on $\wp(\pspace)$ that dominate $\lpr$ on its domain
$\{A\colon A^c \in \pdomain\}$.

The upper envelope $\unex$ of $\solp(\upr)$ is called the
\emph{natural extension} \cite[Thm.~3.4.1]{walley1991}
of $\upr$:
\begin{equation*}
  \unex(A)=\sup\{\pr(A)\colon \pr \in \solp(\upr)\}
\end{equation*}
for all $A\subseteq\pspace$.
The corresponding lower probability is denoted by $\lnex$, so
$\lnex(A)=1-\unex(A^c)$. Clearly, $\lnex$ is the lower envelope of $\solp(\upr)$.

We say that $\upr$ is \emph{coherent} (see \cite[p.~134,
Sec.~3.3.3]{walley1991}) when it coincides with $\unex$ on its domain,
that is, when, for all $A\in\pdomain$,
\begin{equation}\label{eq:coherence}
 \upr(A)=\unex(A).
\end{equation}
The lower
probability $\lpr$ is called {\em coherent} whenever $\upr$
is.

The upper envelope of any set of finitely additive probabilities on
$\wp(\pspace)$ is coherent.
A coherent upper probability $\upr$ and its conjugate lower
probability $\lpr$ satisfy the following properties
\cite[Sec.~2.7.4]{walley1991}, whenever the relevant events
belong to their domain:
\begin{enumerate}
 \item $0 \leq \lpr(A)\leq \upr(A)\leq 1$.
 \item $A \subseteq B$ implies $\upr(A) \leq
 \upr(B)$ and $\lpr(A) \leq \lpr(B)$. [Monotonicity]
 \item $\upr(A \cup B)\leq \upr(A)+\upr(B)$. [Subadditivity]
\end{enumerate}

\subsection{P-Boxes}
\label{sec:charac}

In this section, we revise the theory and some of the main results for p-boxes
defined on totally preordered (not necessarily finite) spaces. For further
details, we refer to \cite{unpub:troffaesdestercke::pboxes:multivariate}.

We start with a totally preordered space $(\pspace,\preceq)$. So, $\preceq$ is transitive and
reflexive and any two elements are comparable. As usual, we write $x\prec y$ for
$x\preceq y$ and $x\not\succeq y$, $x\succ y$ for $y\prec x$, and $x\simeq y$ for $x\preceq y$ and $y\preceq x$. For any two $x$,
$y\in\pspace$ exactly one of $x\prec y$, $x\simeq y$, or $x\succ y$ holds. We also use the following common notation for intervals in $\pspace$:
\begin{align*}
  [x,y]&=\{z\in\pspace\colon x\preceq z\preceq y\} \\
  (x,y)&=\{z\in\pspace\colon x\prec z\prec y\}
\end{align*}
and similarly for $[x,y)$ and $(x,y]$.

We assume that $\pspace$ has a smallest element $0_\pspace$ and a
largest element $1_\pspace$. This is no essential assumption, since
we can always add these two elements to the space $\pspace$.

A \emph{cumulative distribution function} is a non-decreasing map
$\df:\pspace\rightarrow [0,1]$ for which $\df(1_\pspace)=1$. For
each $x\in\pspace$, $\df(x)$ is interpreted as
the probability of $[0_\pspace,x]$. No further restrictions are imposed on $\df$.

The quotient set of $\pspace$ with respect to $\simeq$ is denoted by $\pspace/\simeq$:
\begin{align*}
  [x]_{\simeq}&=\{y\in\pspace\colon y\simeq x\} \text{ for any }x\in\pspace \\
  \pspace/\simeq&=\{[x]_{\simeq}\colon x\in\pspace\}.
\end{align*}
Because $\df$ is non-decreasing, $\df$ is constant on elements
$[x]_{\simeq}$ of $\pspace/\simeq$---we will use this repeatedly.

\begin{definition}
  A \emph{probability box}, or \emph{p-box}, is a pair $(\ldf,\udf)$ of
  cumulative distribution functions from $\pspace$ to $[0,1]$ satisfying $\ldf\leq\udf$.
\end{definition}

A p-box is interpreted as a lower and an upper cumulative
distribution function (see Fig.~\ref{fig:genpbox}), or more
specifically, as an upper probability $\upbox$ on the set of events
\begin{equation*}
  \set{[0_\pspace,x]}{x\in\pspace}\cup\set{(y,1_\pspace]}{y\in\pspace}
\end{equation*}
defined by
\begin{equation}\label{eq:basic-pbox}
  \upbox([0_\pspace,x])=\udf(x)\text{ and }\upbox((y,1_\pspace])=1-\ldf(y).
\end{equation}
We denote by $\unepbox$ the natural extension of $\upbox$ to all
events.

\begin{figure}
\begin{tikzpicture}[scale=0.85]
\tikzstyle{indf}=[circle,draw=black,fill=black,scale=0.3]
\tikzstyle{notin}=[circle,draw=black,solid,scale=0.3]
\tikzstyle{inout}=[circle,draw=gray,fill=gray,scale=0.3]
\draw[->] (-0.1,0) -- (5.5,0) node[below right] {$\pspace$}; \draw
(5,0.1) -- (5,-0.1) node[below] {$1$}; \draw[->] (0,-0.1)
node[below] {$0$} -- (0,3) node[left] {1} -- (0,3.2);
\draw[domain=0:5,thick]  plot[id=cdf1,smooth] function{3*norm((x-2)*2)};
\draw[domain=0:5,thick]  plot[id=cdf2,smooth] function{3*norm((x-3)*2)};
\node at (1.7,1.5) {$\udf$}; \node at (3.3,1.5) {$\ldf$};
\end{tikzpicture}
\caption{Example of a p-box on $[0,1]$.}
\label{fig:genpbox}
\end{figure}
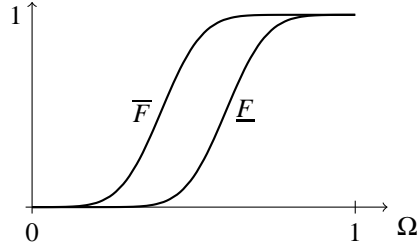

We now recall the main results that we shall need regarding the natural extension
$\unepbox$ of $\upbox$
(see \cite{unpub:troffaesdestercke::pboxes:multivariate} for further details).
First, because $\upbox$ is coherent,
$\unepbox$ coincides with $\upbox$ on its domain.

Next, to simplify the expression for
natural extension, we introduce an element $0_\pspace-$ such that:
\begin{gather*}
  0_\pspace-\prec x\text{ for all }x\in\pspace
  \\
  \df(0_\pspace-)=\ldf(0_\pspace-)=\udf(0_\pspace-)=0.
\end{gather*}
Note that $(0_\pspace-,x]=[0_\pspace,x]$.
Now, let $\pspace^*=\pspace\cup\{0_\pspace-\}$, and define
\begin{equation}\label{eq:plattice}
  \plattice=\{
  (x_0,x_1]\cup(x_2,x_3]\cup\dots\cup(x_{2n},x_{2n+1}]
  \colon
  x_0\prec x_1\prec \dots\prec x_{2n+1}\text{ in }\pspace^*\}.
\end{equation}

\begin{proposition}[\cite{unpub:troffaesdestercke::pboxes:multivariate}]\label{prop:nex-lattice}
  For any $A\in\plattice$, that is
  $A=(x_0,x_1]\cup(x_2,x_3]\cup\dots\cup(x_{2n},x_{2n+1}]$ with $x_0\prec x_1\prec \dots\prec x_{2n+1}$ in $\pspace^*$,
  it holds that $\unepbox(A)=\upboxlattice(A)$, where
  \begin{equation}
    \label{eq:nex-lattice}
    \upboxlattice(A)=
    1-\sum_{k=0}^{n+1} \max\{0,\ldf(x_{2k})-\udf(x_{2k-1})\},
  \end{equation}
  with $x_{-1}=0_\pspace-$ and $x_{2n+2}=1_\pspace$.
\end{proposition}

To calculate $\unepbox(A)$ for an arbitrary event
$A\subseteq\pspace$, we can use the outer measure
\cite[Cor.~3.1.9,p.~127]{walley1991} $\upboxlattice^*$ of the upper
probability $\upboxlattice$ defined in Eq.~\eqref{eq:nex-lattice}:
\begin{equation}\label{eq:outer-measure}
  \unepbox(A)=\upboxlattice^*(A)=\inf_{C\in\plattice,\,A\subseteq C}\upboxlattice(C).
\end{equation}

For intervals, we immediately infer from
Proposition~\ref{prop:nex-lattice} and Eq.~\eqref{eq:outer-measure}
that
\begin{subequations}\label{eq:nex-intervals}
\begin{align}
\unepbox((x,y])&=\udf(y)-\ldf(x)
\\
\unepbox([x,y])&=\udf(y)-\ldf(x-)
\\
\unepbox((x,y))&=
\begin{cases}
  \udf(y)-\ldf(x) & \text{ if $y$ has no immediate predecessor}
  \\
  \udf(y-)-\ldf(x) & \text{ if $y$ has an immediate predecessor}
\end{cases}
\\
\unepbox([x,y))&=
\begin{cases}
  \udf(y)-\ldf(x-) & \text{ if $y$ has no immediate predecessor}
  \\
  \udf(y-)-\ldf(x-) & \text{ if $y$ has an immediate predecessor}
\end{cases}
\end{align}
\end{subequations}
for any $x\prec y$ in $\pspace$,\footnote{In case $x=0_\pspace$,
evidently, $0_\pspace-$ is the immediate predecessor.} where
$\udf(y-)$ denotes $\sup_{z\prec y}\udf(z)$ and similarly for
$\ldf(x-)$. If $\pspace/\simeq$ is finite, then one can think of
$z-$ as the immediate predecessor of $z$ in the quotient space
$\pspace/\simeq$ for any $z \in \pspace$. Note that in particular
\begin{equation}\label{eq:unex-singleton}
 \unepbox(\{x\})=\udf(x)-\ldf(x-)
\end{equation}
for any $x \in \pspace$. We will use this repeatedly.

\subsection{Possibility and Maxitive Measures}

Very briefly, we introduce possibility and maxitive measures. For
further information, see
\cite{zadeh1978a,dubois1988,walley1996,cooman1998}.

\begin{definition}
 A \emph{maxitive measure} is an upper probability $\upr\colon\wp(\pspace)\to[0,1]$ satisfying $\upr(A \cup
 B)=\max\{\upr(A),\upr(B)\}$ for every $A$, $B\subseteq\pspace$.
\end{definition}

It follows from the above definition that a maxitive measure is also
maximum-preserving when we consider finite unions of events.

The following result is well-known, but we include a quick proof for
the sake of completeness.

\begin{proposition}
  A maxitive measure $\upr$ is coherent whenever $\upr(\emptyset)=0$ and $\upr(\pspace)=1$.
\end{proposition}
\begin{proof}
By \cite[Theorem~1]{nguyen1997}, a maxitive measure $\upr$
satisfying $\upr(\emptyset)=0$ is $\infty$-alternating, and as a
consequence also $2$-alternating.
Whence, $\upr$ is coherent by \cite[p.~55,
Corollary~6.3]{walley1981}.
\end{proof}

Possibility measures are a particular case of maxitive measures.

\begin{definition}
  A (normed) \emph{possibility distribution} is a mapping
  $\pi\colon\pspace\to[0,1]$ satisfying $\sup_{x\in\pspace}\pi(x)=1$.
  A possibility distribution $\pi$ induces a \emph{possibility measure} $\Pi$ on $\wp(\pspace)$, given by:
  \begin{equation*}
    \Pi(A)=\sup_{x\in A}\pi(x)\text{ for all }A\subseteq\pspace.
  \end{equation*}
\end{definition}

Equivalently, possibility measures can be defined as
supremum-preserving upper probabilities, i.e., as functionals $\Pi$
for which
\begin{equation*}
 \Pi(\cup_{A \in {\mathcal A}} A)=\sup_{A \in {\mathcal A}} \Pi(A)
 \ \forall {\mathcal A} \subseteq {\mathcal P}(\pspace).
\end{equation*}

If we write $\lnepossib$ for the conjugate lower probability of the upper
probability $\Pi$, then:
\begin{equation*}
  \lnepossib(A)=1-\Pi(A^c)=1-\sup_{x\in A^c}\pi(x). 
\end{equation*}

A possibility measure is maxitive, but not
all maxitive measures are possibility measures.

As an imprecise probability model, possibility measures are not as
expressive as for instance p-boxes---for example, the only
probability measures that can be represented by possibility measures
are the degenerate ones. This poor expressive power is also
illustrated by the fact that, for any event $A$:
\begin{align*}
\Pi(A)<1 &\implies \lnepossib(A)=0,
&&\text{ and }
\\
\lnepossib(A)>0 &\implies \Pi(A)=1,
\end{align*}
meaning that
every event has a trivial probability bound on at least one side. Their main
attraction is that calculations with them are very easy: to find the
upper (or lower) probability of any event, a simple supremum
suffices.

In the following sections, we characterize the circumstances under
which a possibility measure $\Pi$ is the natural extension of some p-box
$(\ldf,\udf)$. In order to do so, we first characterise the
conditions under which a p-box induces a maxitive measure.

\section{P-boxes as Maxitive Measures.}
\label{sec:pbox-as-max}

We show here that p-boxes $(\ldf,\udf)$ on any totally preordered space
where at least one of $\ldf$ or $\udf$ is $0$--$1$-valued are maxitive measures, and in this sense are closely related to possibility measures. We then derive a simple closed expression of the (upper)
natural extension of such p-boxes.

\subsection{A Necessary Condition for Maxitivity.}

\begin{proposition}\label{prop:pbox-possibility-zero-one}
 If the natural extension $\unepbox$ of a p-box $(\ldf,\udf)$
 is maximum preserving, then at least one of $\ldf$ or
 $\udf$ is $0$--$1$-valued.
\end{proposition}
\begin{proof}

We begin by showing that there is no $x \in \pspace$ such that
$0<\ldf(x)\leq \udf(x)<1$.
Assume ex absurdo that there is such an
$x$. For $\unepbox$ to be maximum preserving, we require that
\begin{equation*}
  \unepbox([0_\pspace,x]\cup (x,1_\pspace])=\max\{\unepbox([0_\pspace,x]),\unepbox((x,1_\pspace])\}
\end{equation*}
But this cannot be. The left hand side is $1$, whereas the right
hand side is strictly less than one, because, by
Eq.~\eqref{eq:basic-pbox},
\begin{align*}
  \unepbox([0_\pspace,x])&=\udf(x)<1, \\
  \unepbox((x,1_\pspace])&=1-\ldf(x)<1.
\end{align*}

Whence, for every $x\in\pspace$, at least one of $\ldf(x)=0$ or
$\udf(x)=1$ must hold. In other words, $\ldf(x)=0$ whenever
$\udf(x)<1$, and $\udf(x)=1$ whenever $\ldf(x)>0$ (see
Figure~\ref{fig:pbox-possibility-zero-one}).
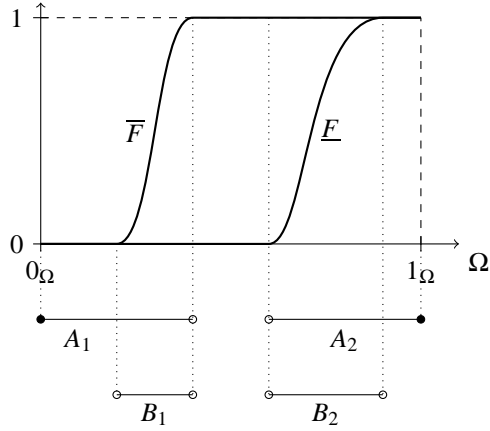
\begin{figure}
\begin{center}
\begin{tikzpicture}
\tikzstyle{blackdot}=[circle,draw=black,fill=black,scale=0.3]
\tikzstyle{opendot}=[circle,draw=black,solid,scale=0.3]
\draw[->] (-0.1,0) node[left] {$0$} -- (5.5,0) node[below right] {$\pspace$};
\draw (5,0.1) -- (5,-0.1) node[below] {$1_\pspace$};
\draw[->] (0,-0.1) node[below] {$0_\pspace$} -- (0,3.2);
\draw (0.1,3) -- (-0.1,3) node[left] {$1$};
\draw[thick] (0,0) -- (3,0) .. controls (3.5,0) and (3.5,3) .. (4.5,3) node[midway,right] {$\ldf$} -- (5,3);
\draw[thick] (0,0) -- (1,0) .. controls (1.5,0) and (1.5,3)  .. (2,3) node[midway,left] {$\udf$}  -- (5,3) ;
\draw[dashed] (5,0) -- (5,3);
\draw[dashed] (0,3) -- (5,3);
\draw (0,-1) node[blackdot]{} -- node[below,near start]{$A_1$} (2,-1) node[opendot]{};
\draw (3,-1) node[opendot]{} -- node[below]{$A_2$} (5,-1) node[blackdot]{};
\draw (1,-2) node[opendot]{} -- node[below]{$B_1$} (2,-2) node[opendot]{};
\draw (3,-2) node[opendot]{} -- node[below]{$B_2$} (4.5,-2) node[opendot]{};
\draw[dotted] (0,0) -- (0,-1);
\draw[dotted] (1,0) -- (1,-2);
\draw[dotted] (2,3) -- (2,-2);
\draw[dotted] (3,3) -- (3,-2);
\draw[dotted] (4.5,3) -- (4.5,-2);
\draw[dotted] (5,0) -- (5,-1);
\end{tikzpicture}
\caption{A p-box for the proof of Proposition~\ref{prop:pbox-possibility-zero-one}.}
\label{fig:pbox-possibility-zero-one}
\end{center}
\end{figure}
Hence, the sets
\begin{align*}
  A_1&:=\{x \in \pspace: \udf(x)<1\} \\
  A_2&:=\{y \in \pspace: \ldf(y)>0\}
\end{align*}
are disjoint, and $A_1\prec A_2$ in the sense that $x\prec y$ for
all $x\in A_1$ and $y\in A_2$. Indeed, if $x\in A_1$ and $y\in A_2$,
then $\udf(x)<1$, and $\udf(y)=1$ because $\ldf(y)>0$. These can
only hold simultaneously if $x\prec y$.

Note that $A_1$ is empty when $\udf(x)=1$ for all $x\in\pspace$, and
in this case the desired result is trivially established. $A_2$ is
non-empty because $\ldf(1_\pspace)=1$. Anyway, consider the sets
\begin{align*}
  B_1&:=\{x \in \pspace: 0<\udf(x)<1\}\subseteq A_1 \\
  B_2&:=\{y \in \pspace: 0<\ldf(y)<1\}\subseteq A_2
\end{align*}
The proposition is established if we can show that at least one of
these two sets is empty.

Suppose, ex absurdo, that both are non-empty. Pick any element $c
\in B_1$ and $d \in B_2$ and consider the set $C=[0_\pspace,c] \cup
(d,1_\pspace]$---note that $c\prec d$ because $c\in A_1$ and $d\in
A_2$, so $(c,d]$ is non-empty.
Whence, by Eq.~\eqref{eq:nex-lattice},
\begin{equation*}
 \unepbox([0_\pspace,c] \cup (d,1_\pspace])
 =1-\max\{0,\ldf(d)-\udf(c)\}.
\end{equation*}
Also, by Eq.~\eqref{eq:basic-pbox},
\begin{align*}
  \unepbox([0_\pspace,c])&=\udf(c), \\
  \unepbox((d,1_\pspace])&=1-\ldf(d).
\end{align*}
So, for $\unepbox$ to be maximum preserving, we require that
\begin{align*}
  1-\max\{0,\ldf(d)-\udf(c)\}=\max\{\udf(c),1-\ldf(d)\}.
\end{align*}
But this cannot hold. Indeed, because $0<\udf(c)<1$ and
$0<1-\ldf(d)<1$,
the above equality can only hold if $\ldf(d)-\udf(c)>0$---otherwise
the left hand side would be $1$ whereas the right hand side is
strictly less than $1$. So, effectively, we require that
\begin{align*}
  1-\ldf(d)+\udf(c)=\max\{\udf(c),1-\ldf(d)\}.
\end{align*}
This cannot hold, because the sum of two strictly positive numbers
(in this case $1-\ldf(d)$ and $\udf(c)$) is always strictly larger
than their maximum. We conclude that $\unepbox$
cannot be maximum preserving
if both $B_1$ and $B_2$ are
non-empty. In other words, at least one of $\ldf$ or $\udf$ must be
$0$--$1$-valued.
\end{proof}

\subsection{Sufficient Conditions for Maxitivity.} We derive sufficient conditions for the two different cases described by Proposition~\ref{prop:pbox-possibility-zero-one}, starting with  $0$--$1$-valued $\ldf$.

\subsubsection{Maxitivity for Zero-One Valued Lower Cumulative Distribution Functions.}

We first provide a simple expression for the natural extension of such p-boxes over events.

\begin{proposition}\label{prob:unex-0-1-ldf}
Let $(\ldf,\udf)$ be a p-box with $0$--$1$-valued $\ldf$, and
let $B=\{x \in \pspace^*\colon \ldf(x)=0\}$. Then, for any $A
\subseteq \pspace$,
\begin{align}
  \label{eq:unex-0-1-ldf:cases}
  \unepbox(A)
  &=
  \begin{cases}
    \inf_{x \in \pspace^*\colon A \cap B \preceq x} \udf(x)
    & \text{if }y\prec A\cap B^c\text{ for at least one }y\in B^c,
    \\
    1
    & \text{otherwise}.
  \end{cases}
  \\
  \label{eq:unex-0-1-ldf:min}
  &=
  \min_{y\in B^c}\inf_{x \in \pspace^*\colon A \cap [0_\pspace,y]\preceq x} \udf(x).
\end{align}
\end{proposition}

In the above, $A\preceq x$ means $z\preceq x$ for all $z\in A$, and
similarly $y\prec A$ means $y\prec z$ for all $z\in A$. For example,
it holds that $\emptyset\preceq x$ and $y\prec\emptyset$ for all $x$ and $y$.

\begin{proof}
We deduce from Eq.~\eqref{eq:outer-measure} and from the conjugacy
between $\lnepbox$ and $\unepbox$ that for any $A\subseteq\pspace$,
\begin{equation*}
  \lnepbox(A)
  =\sup_{(x_0,x_1]\cup\dots\cup(x_{2n},x_{2n+1}]\subseteq
  A}\sum_{k=0}^{n}\max\{0,\ldf(x_{2k+1})-\udf(x_{2k})\}.
\end{equation*}
All the terms in this sum are zero except possibly for one (if it
exists) where $x_{2k} \in B, x_{2k+1} \in B^c$, where we get
$1-\udf(x_{2k})$. Aside, as subsets of $\pspace^*$, note that both
$B$ and $B^c$ are non-empty: $0_\pspace-\in B$ and $1_\pspace\in
B^c$. Consequently,
\begin{align*}
  \lnepbox(A)&=
   1-\inf_{x,y\colon x\in B,\,y\in B^c,\,(x,y]\subseteq A}\udf(x);
\end{align*}
and therefore
\begin{align*}
  \unepbox(A)
  &=
  \inf_{x,y\colon x\in B,\,y\in B^c,\,(x,y]\subseteq A^c}\udf(x)
  \\
  &=
  \inf_{x,y\colon x\in B,\,y\in B^c,\,A\subseteq [0_\pspace,x]\cup(y,1_\pspace]}\udf(x)
\end{align*}
where it is understood that the infimum evaluates to $1$ whenever
there are no $x\in B$ and $y\in B^c$ such that $A\subseteq
[0_\pspace,x]\cup(y,1_\pspace]$.

Now, for any $x \in B$ and $y \in B^c$, it holds that $A\subseteq
[0_\pspace,x]\cup(y,1_\pspace]$ if and only if
\begin{gather*}
A\cap B\subseteq ([0_\pspace,x]\cup(y,1_\pspace])\cap B=[0_\pspace,x]
\text{ and}
\\
A\cap B^c\subseteq ([0_\pspace,x]\cup(y,1_\pspace])\cap B^c=(y,1_\pspace],
\end{gather*}
that is, if and only if
\begin{gather*}
A\cap B\preceq x
\text{ and }
y\prec A\cap B^c.
\end{gather*}
Hence, if there is an $y \in B^c$ such that
$y\prec A\cap B^c$, then:
\begin{enumerate}[(i)]
\item either there is no $x \in B$ such that $A \cap B\preceq x$, whence
\begin{equation*}
 \unepbox(A)=1=\inf_{x \in \pspace^*\colon A \cap B\preceq x}
 \udf(x),
\end{equation*}
taking into account that for any $x \in \pspace^*$ such that $A \cap B\preceq x$
it must be that $x \in B^c$, whence
$\ldf(x)=\udf(x)=1$;
\item or there is some $x \in B$
such that $A \cap B\preceq x$, in which case
\begin{equation*}
  \unepbox(A)
  =\inf_{x \in B\colon A \cap B\preceq x}\udf(x)
  =\inf_{x \in \pspace^*\colon A \cap B\preceq x}\udf(x),
\end{equation*}
where the second equality follows from the monotonicity of $\udf$.
\end{enumerate}
This establishes Eq.~\eqref{eq:unex-0-1-ldf:cases}.

We now turn to proving Eq.~\eqref{eq:unex-0-1-ldf:min}.
In case $y\prec A\cap B^c$ for at least one
$y\in B^c$, it follows that
\begin{equation*}
  \unepbox(A)
  =\inf_{x \in \pspace^*\colon A \cap B\preceq x}\udf(x)
\end{equation*}
But in this case, $A\cap B=A\cap[0_\pspace,y']$ for any $y'\in B^c$
such that $y'\preceq y$, because
\begin{equation*}
  A\cap [0,y']=A\cap [0_\pspace,y']\cap(B\cup B^c)
  =(A\cap B\cap [0_\pspace,y'])\cup (A\cap B^c\cap [0_\pspace,y'])
  =A\cap B
\end{equation*}
as $B\cap [0_\pspace,y']=B$ and $A\cap B^c\cap
[0_\pspace,y']=\emptyset$ because $y'\preceq y$ and $y\prec A\cap
B^c$. So, by the monotonicity of $\udf$,
Eq.~\eqref{eq:unex-0-1-ldf:min} follows.

In case $y\not\prec A\cap B^c$ for all
$y\in B^c$, it follows that
\begin{equation*}
  \unepbox(A)
  =1
  =\udf(x)
\end{equation*}
for all $x$ in $B^c$---indeed, because $A\cap [0,y]\cap
B^c\neq\emptyset$ for every $y\in B^c$, it holds that $A\cap [0,y]\preceq x$
implies $x \in B^c$, and hence $\udf(x)=1$.
Again, Eq.~\eqref{eq:unex-0-1-ldf:min} follows.
\end{proof}

A few common important special cases are summarized in the following corollary:

\begin{corollary}
  Let $(\ldf,\udf)$ be a p-box with $0$--$1$-valued $\ldf$, and let
  $B=\{x \in \pspace^*\colon \ldf(x)=0\}$. If
  $\pspace/\simeq$ is order complete,
  then, for any $A \subseteq \pspace$,
  \begin{equation}\label{eq:unex-0-1-ldf-order-complete}
    \unepbox(A)=\min_{y\in B^c}\udf(\sup A\cap [0_\pspace,y]).
  \end{equation}
  If, in addition, $B^c$ has a minimum, then
  \begin{equation}\label{eq:unex-0-1-ldf-order-complete-minimum}
    \unepbox(A)=\udf(\sup A\cap [0_\pspace,\min B^c]).
  \end{equation}
  If, in addition, $B^c=[1_\pspace]_\simeq$ (this occurs exactly when
  $\ldf$ is vacuous, i.e. $\ldf=I_{[1_\pspace]_\simeq}$), then
  \begin{equation}\label{eq:unex-0-1-ldf-order-complete-vacuous}
    \unepbox(A)=\udf(\sup A).
  \end{equation}
\end{corollary}

Note that Eq.~\eqref{eq:unex-0-1-ldf-order-complete-vacuous} is
essentially due to \cite[paragraph preceeding Theorem~11]{cooman1998}---they work
with chains and multivalued mappings, whereas we work with total preorders. We are now ready to show that the considered p-boxes are maxitive measures.

\begin{proposition}\label{prop:unex-0-1-ldf-max}
Let $(\ldf,\udf)$ be a p-box where $\ldf$ is $0$--$1$-valued.
Then $\unepbox$ is maximum-preserving.
\end{proposition}
\begin{proof}
Consider a finite collection $\mathcal{A}$ of subsets of $\pspace$.
If there are $A \in \mathcal{A}$ such that, for all $y \in B^c$, $y\not\prec A\cap B^c$, then
$\unepbox(A)=1=\unepbox(\cup_{A \in \mathcal{A}} A)$
by Eq.~\eqref{eq:unex-0-1-ldf:cases},
establishing the desired result for this case.

So, from now on, we may assume that, for every $A \in \mathcal{A}$, there is
a $y_A\in B^c$ such that $y_A\prec A\cap B^c$.
With $y=\min_{A \in\mathcal{A}} y_A \in B^c$,
it holds that
$y\prec \cup_{A \in \mathcal{A}} A\cap B^c$,
and so, by
Eq.~\eqref{eq:unex-0-1-ldf:cases},
\begin{align*}
  \unepbox(A)
  &=\inf_{x \in \pspace^*\colon A \cap B\preceq x} \udf(x)
  \text{ for every } A \in \mathcal{A}, \text{ and }\\
  \unepbox(\cup_{A \in \mathcal{A}}A)
  &=\inf_{x \in \pspace^*\colon \cup_{A \in \mathcal{A}} A \cap B \preceq x} \udf(x).
\end{align*}

Now, because $\mathcal{A}$ is finite,
there is an $A'\in\mathcal{A}$ such
that
\begin{align*}
  \{x \in \pspace^*\colon A' \cap B \preceq x\}
  &=
  \cap_{A \in\mathcal{A}} \{x \in \pspace^*\colon A \cap B \preceq x\}
  \\
  \intertext{and because
    $\cup_{A\in\mathcal{A}} A \cap B \preceq x$ if and only if $A\cap B\preceq x$ for
    all $A\in\mathcal{A}$,}
  &=
  \{x \in \pspace^*\colon \cup_{A \in \mathcal{A}} A\cap B\preceq x\}.
\end{align*}
Consequently,
\begin{multline*}
  \max_{A\in\mathcal{A}}\unepbox(A)
  =\max_{A\in\mathcal{A}}\inf_{x \in \pspace^*\colon A \cap B\preceq x}\udf(x)
  \\
  \geq \inf_{x\in \pspace^*\colon A' \cap B\preceq x}\udf(x)
  =\inf_{x\in \pspace^*\colon \cup_{A\in\mathcal{A}}A \cap B\preceq x}\udf(x)
  =\unepbox(\cup_{A\in\mathcal{A}}A).
\end{multline*}
The converse inequality follows from the
coherence of $\unepbox$. Concluding,
\begin{equation*}
  \max_{A\in\mathcal{A}}\unepbox(A)=\unepbox\left(\cup_{A\in\mathcal{A}}A\right)
\end{equation*}
for any finite collection $\mathcal{A}$ of subsets of $\pspace$.
\end{proof}

\subsubsection{Maxitivity for Zero-One Valued Upper Cumulative Distribution Functions.}

Let us now consider the case of $0$--$1$-valued $\udf$.

\begin{proposition}\label{prob:unex-0-1-udf}
Let $(\ldf,\udf)$ be a p-box with $0$--$1$-valued $\udf$, and
let $C=\{x \in \pspace^*\colon \udf(x)=0\}$. Then, for any $A
\subseteq \pspace$,
\begin{align}
  \label{eq:unex-0-1-udf:cases}
  \unepbox(A)
  &=
  \begin{cases}
    1-\sup_{y \in \pspace^*\colon y\prec A \cap C^c} \ldf(y)
    & \text{if }A\cap C\preceq x\text{ for at least one }x\in C,
    \\
    1
    & \text{otherwise}.
  \end{cases}
  \\
  \label{eq:unex-0-1-udf:max}
  &=
  1-\max_{x \in C}\sup_{y \in \pspace^*\colon y\prec A \cap (x,1_\pspace]} \ldf(y).
\end{align}
\end{proposition}

\begin{proof}
We deduce from Eq.~\eqref{eq:outer-measure} and from the conjugacy
between $\lnepbox$ and $\unepbox$ that for any $A\subseteq\pspace$,
\begin{equation*}
  \lnepbox(A)
  =\sup_{(x_0,x_1]\cup\dots\cup(x_{2n},x_{2n+1}]\subseteq
  A}\sum_{k=0}^{n}\max\{0,\ldf(x_{2k+1})-\udf(x_{2k})\}.
\end{equation*}
All the terms in this sum are zero except possibly for one (if it
exists) where $x_{2k} \in C, x_{2k+1} \in C^c$, where we get
$\ldf(x_{2k+1})$. Aside, as subsets of $\pspace^*$, note that both
$C$ and $C^c$ are non-empty: $0_\pspace-\in C$ and $1_\pspace\in
C^c$. Consequently,
\begin{align*}
  \lnepbox(A)&=
   \sup_{x,y\colon x\in C,\,y\in C^c,\,(x,y]\subseteq A}\ldf(y);
\end{align*}
and therefore
\begin{align*}
  \unepbox(A)
  &=
  1-\sup_{x,y\colon x\in C,\,y\in C^c,\,(x,y]\subseteq A^c}\ldf(y)
  \\
  &=
  1-\sup_{x,y\colon x\in C,\,y\in C^c,\,A\subseteq [0_\pspace,x]\cup(y,1_\pspace]}\ldf(y)
\end{align*}
where it is understood that the supremum evaluates to $0$ whenever
there are no $x\in C$ and $y\in C^c$ such that $A\subseteq
[0_\pspace,x]\cup(y,1_\pspace]$.

Now, for any $x \in C$ and $y \in C^c$, it holds that $A\subseteq
[0_\pspace,x]\cup(y,1_\pspace]$ if and only if
\begin{gather*}
A\cap C\subseteq ([0_\pspace,x]\cup(y,1_\pspace])\cap
C=[0_\pspace,x] \text{ and}
\\
A\cap C^c\subseteq ([0_\pspace,x]\cup(y,1_\pspace])\cap C^c=(y,1_\pspace],
\end{gather*}
that is, if and only if
\begin{gather*}
A\cap C\preceq x
\text{ and }
y\prec A\cap C^c.
\end{gather*}
Hence, if there is an $x \in C$ such that
$A\cap C\preceq x$, then:
\begin{enumerate}[(i)]
\item either there is no $y \in C^c$ such that $y\prec A \cap C^c$, whence
\begin{equation*}
 \unepbox(A)=1=1-\sup_{y \in \pspace^*\colon y\prec A \cap C^c}
 \ldf(y),
\end{equation*}
taking into account that for any $y \in \pspace^*$ such that $y\prec A \cap C^c$
it must be that $y \in C$, whence
$\ldf(y)=\udf(y)=0$;
\item or there is some $y \in C^c$
such that $y\prec A \cap C^c$, in which case
\begin{equation*}
  \unepbox(A)
  =1-\sup_{y \in C^c\colon y\prec A \cap C^c}\ldf(y)
  =1-\sup_{y \in \pspace^*\colon y\prec A \cap C^c}\ldf(y),
\end{equation*}
where the second equality follows from the monotonicity of $\ldf$.
\end{enumerate}
This establishes Eq.~\eqref{eq:unex-0-1-udf:cases}.

We now turn to proving Eq.~\eqref{eq:unex-0-1-udf:max}. In case
$A\cap C\preceq x$ for at least one $x\in C$, it follows that
\begin{equation*}
  \unepbox(A)
  =1-\sup_{y \in \pspace^*\colon y\prec A \cap C^c}\ldf(y).
\end{equation*}
But in this case, $A\cap C^c=A\cap(x',1_\pspace]$ for any $x'\in C$
such that $x'\succeq x$, because
\begin{equation*}
  A\cap (x',1_\pspace]=A\cap (x',1_\pspace]\cap(C\cup C^c)
  =(A\cap C\cap (x',1_\pspace])\cup (A\cap C^c\cap (x',1_\pspace])
  =A\cap C^c
\end{equation*}
as $C^c\cap (x',1_\pspace]=C^c$ and $A\cap C\cap
(x',1_\pspace]=\emptyset$ by assumption. So, by the monotonicity of
$\ldf$, Eq.~\eqref{eq:unex-0-1-udf:max} follows.

In case $A\cap C\not\preceq x$ for all
$x\in C$, it follows that
\begin{equation*}
  \unepbox(A)
  =1
  =1-\ldf(y)
\end{equation*}
for all $y$ in $C$---indeed, because $A\cap (x,1_\pspace]\cap
C\neq\emptyset$ for every $x\in C$, it holds that $y\prec A\cap(x,1_\pspace]$
implies $y\in C$, and hence $\ldf(y)=0$.
Again, Eq.~\eqref{eq:unex-0-1-udf:max} follows.
\end{proof}

A few common important special cases are summarized in the following corollary:

\begin{corollary}\label{cor:unex-0-1-udf:order:complete}
  Let $(\ldf,\udf)$ be a p-box with $0$--$1$-valued $\udf$, and let
  $C=\{x \in \pspace^*\colon \udf(x)=0\}$. If
  $\pspace/\simeq$ is order complete,
  and $C$ has a maximum,
  then, for any $A \subseteq \pspace$,
\begin{equation}\label{eq:unex-0-1-udf-order-complete-maximum}
  \unepbox(A)=
  \begin{cases}
   1-\ldf(\inf A\cap C^c) &\text{if } A\cap C^c \text{ has no minimum}\\
   1-\ldf((\min A\cap C^c)-) &\text{if } A\cap C^c \text{ has a  minimum}.
  \end{cases}
\end{equation}
  If, in addition, $C=\{0_\pspace-\}$ (this occurs exactly when
  $\udf$ is vacuous, i.e. $\udf=1$), then
  \begin{equation}\label{eq:unex-0-1-udf-order-complete-vacuous}
    \unepbox(A)=
  \begin{cases}
    1-\ldf(\inf A)&\text{if } A \text{ has no minimum} \\
    1-\ldf(\min A-)&\text{if } A \text{ has a minimum}.
  \end{cases}
  \end{equation}
\end{corollary}
\begin{proof}
  Use Proposition~\ref{prob:unex-0-1-udf}, and note that $(\max C,1_\pspace]=C^c$.
\end{proof}

Using Eq.~\eqref{eq:unex-0-1-udf:cases}, we can also show that $\unepbox$ is
maximum-preserving when $\udf$ is $0$--$1$-valued:

\begin{proposition}\label{prop:unex-0-1-udf-max}
Let $(\ldf,\udf)$ be a p-box where $\udf$ is $0$--$1$-valued. Then
$\unepbox$ is maximum-preserving.
\end{proposition}
\begin{proof}
Consider a finite collection $\mathcal{A}$ of subsets of $\pspace$.
If there are $A \in \mathcal{A}$ such that, for all $x\in C$, $A\cap C\not\preceq x$, then
$\unepbox(A)=1=\unepbox(\cup_{A \in \mathcal{A}} A)$
by Eq.~\eqref{eq:unex-0-1-udf:cases},
establishing the desired result for this case.

So, from now on, we may assume that, for every $A \in \mathcal{A}$, there is
an $x_A\in C$ such that $A\cap C\preceq x_A$.
With $x=\max_{A \in\mathcal{A}} x_A \in C$,
it holds that
$\cup_{A \in \mathcal{A}} A\cap C\preceq x$,
and so, by
Eq.~\eqref{eq:unex-0-1-udf:cases},
\begin{align*}
  \unepbox(A)
  &=1-\sup_{y \in \pspace^*\colon y\prec A \cap C^c} \ldf(y)
  \text{ for every } A \in \mathcal{A}, \text{ and }\\
  \unepbox(\cup_{A \in \mathcal{A}}A)
  &=1-\sup_{y \in \pspace^*\colon y\prec \cup_{A\in\mathcal{A}}A \cap C^c} \ldf(y).
\end{align*}

Now, because $\mathcal{A}$ is finite,
there is an $A'\in\mathcal{A}$ such
that
\begin{align*}
  \{y \in \pspace^*\colon y\prec A' \cap C^c\}
  &=
  \cap_{A \in\mathcal{A}} \{y \in \pspace^*\colon y\prec A \cap C^c\}
  \\
  \intertext{and because
    $y\prec \cup_{A\in\mathcal{A}} A \cap C^c$ if and only if $y\prec A\cap C^c$ for
    all $A\in\mathcal{A}$,}
  &=
  \{y \in \pspace^*\colon y\prec \cup_{A\in\mathcal{A}}A \cap C^c\}.
\end{align*}
Consequently,
\begin{multline*}
  \max_{A\in\mathcal{A}}\unepbox(A)
  =\max_{A\in\mathcal{A}}\left(1-\sup_{y \in \pspace^*\colon y\prec A \cap C^c}\ldf(y)\right)
  \\
  \geq 1-\sup_{y\in \pspace^*\colon y\prec A' \cap C^c}\ldf(y)
  =1-\sup_{y\in \pspace^*\colon y\prec \cup_{A\in\mathcal{A}}A \cap C^c}\ldf(y)
  =\unepbox(\cup_{A\in\mathcal{A}}A).
\end{multline*}
The converse inequality follows from the
coherence of $\unepbox$. Concluding,
\begin{equation*}
  \max_{A\in\mathcal{A}}\unepbox(A)=\unepbox\left(\cup_{A\in\mathcal{A}}A\right)
\end{equation*}
for any finite collection $\mathcal{A}$ of subsets of $\pspace$.
\end{proof}

\subsection{Summary of Necessary and Sufficient Conditions} Putting Propositions~\ref{prop:pbox-possibility-zero-one},~\ref{prop:unex-0-1-ldf-max} and~\ref{prop:unex-0-1-udf-max} together, we get the following conditions.

\begin{corollary}\label{cor:max-suf-nec}
Let $(\ldf,\udf)$ be a p-box. Then, $(\ldf,\udf)$ is maximum-preserving if and only if
$$\ldf \textrm{ is  $0$--$1$-valued} $$
or
$$\udf \textrm{ is $0$--$1$-valued}.$$
\end{corollary}

These simple conditions characterise maximum-preserving p-boxes and bring us closer to p-boxes that are possibility measures, and that we will now study.

\section{P-Boxes as Possibility Measures.}
\label{sec:pbox-to-poss}

In this section, we identify when p-boxes coincide exactly with a
possibility measure. By Corollary~\ref{cor:max-suf-nec}, when $\pspace/\simeq$ is finite, $(\ldf,\udf)$ is a possibility
measure if and only if either $\ldf$ or $\udf$ is $0$--$1$-valued.
More generally, when $\pspace/\simeq$ is not finite,
we will rely on the following trivial, yet important, lemma:

\begin{lemma}\label{lem:pbox:is:a:possib}
  For a p-box $(\ldf,\udf)$ there is a
  possibility measure $\Pi$ such that $\unepbox=\Pi$ if and
  only if
  \begin{equation}\label{eq:charac-pbox-possib}
    \unepbox(A)=\sup_{x\in A}\unepbox(\{x\})\text{ for all }A\subseteq\pspace
  \end{equation}
  and in such a case, the possibility distribution $\pi$ associated with $\Pi$ is $\pi(x)=\unepbox(\{x\})$.
\end{lemma}
\begin{proof}
  ``if''. If $\unepbox(A)=\sup_{x\in A}\unepbox(\{x\})$ for all
  $A\subseteq\pspace$, then $\lnepbox=\lnepossib$ with the suggested
  choice of $\pi$, because, for all $A\subseteq\pspace$,
  \begin{equation*}
    \lnepbox(A)=1-\unepbox(A^c)=1-\sup_{x\in A^c}\unepbox(\{x\})=1-\sup_{x\in A^c}\pi(x)=1-\Pi(A^c)=\lnepossib(A).
  \end{equation*}

  ``only if''. If $\lnepbox=\lnepossib$, then, for all
  $A\subseteq\pspace$,
  \begin{equation*}
    \unepbox(A)
    =\Pi(A)
    =\sup_{x\in A}\pi(x)
    =\sup_{x\in A}\Pi(\{x\})
    =\sup_{x\in A}\unepbox(\{x\}).
  \end{equation*}
\end{proof}

We will say that \emph{a p-box $(\ldf,\udf)$ is a possibility
measure} whenever Eq.~\eqref{eq:charac-pbox-possib} is
satisfied.

Note that, due to Proposition~\ref{prop:pbox-possibility-zero-one}, for a p-box to be a possibility measure, at
least one of $\ldf$ or $\udf$ must be $0$--$1$-valued. Next, we give a
characterisation of p-boxes inducing a possibility measure in each
of these two cases.

\subsection{P-Boxes with  Zero-One-Valued Lower Cumulative Distribution Functions}

As mentioned, by Corollary~\ref{cor:max-suf-nec}, a p-box with $0$--$1$-valued $\ldf$ is maxitive, and its upper natural extension is given by Proposition~\ref{prob:unex-0-1-ldf}. Whence, we can easily determine when such p-box is a possibility measure:

\begin{proposition}\label{prop:pbox-iff-possibility-0-1-ldf}
  Assume that $\pspace/\simeq$ is order complete.
  Let $(\ldf,\udf)$ be a p-box with $0$--$1$-valued $\ldf$, and let
  $B=\{x\in\pspace^*\colon\ldf(x)=0\}$. Then, $(\ldf,\udf)$ is a possibility
  measure if and only if
  \begin{enumerate}[(i)]
  \item\label{prop:pbox-iff-possibility-0-1-ldf:xminus} $\udf(x)=\udf(x-)$ for all $x\in\pspace$ that have no immediate predecessor, and
  \item\label{prop:pbox-iff-possibility-0-1-ldf:minimum} $B^c$ has a minimum,
  \end{enumerate}
  and in such a case,
  \begin{equation}\label{eq:pbox-iff-possibility-0-1-ldf}
    \unepbox(A)=\sup_{x\in A\cap[0_\pspace,\min B^c]}\udf(x)
  \end{equation}
\end{proposition}

Note that, in case $1_\pspace$ is a minimum of $B^c$, condition~\eqref{prop:pbox-iff-possibility-0-1-ldf:xminus} is essentially due to
\cite[Observation~9]{cooman1998}. Also note that,
for $\unepbox$ to be a possibility measure, both conditions are still necessary
even when $\pspace/\simeq$ is not order complete: the proof in this direction does not require order completeness.

As a special case, we mention that $\unepbox$ is a
possibility measure with possibility distribution
\begin{equation*}
  \pi(x)=
  \begin{cases}
    \udf(x) & \text{if }x\preceq\min B^c \\
    0 & \text{otherwise}.
  \end{cases}
\end{equation*}
whenever $\pspace/\simeq$ is finite.

\begin{proof}
  ``only if''. Assume that $(\ldf,\udf)$ is a possibility measure.
  For every $x\in\pspace$ that has no immediate predecessor,
  \begin{align*}
    \udf(x-)
    &=\sup_{x'\prec x}\udf(x')
    \\
    \intertext{and
    because $\unepbox(\{x'\})=\udf(x')-\ldf(x'-)$ (see Eq.~\eqref{eq:unex-singleton}),
    }
    &\ge\sup_{x'\prec x}\unepbox(\{x'\})
    \\
    \intertext{and because $(\ldf,\udf)$ is a possibility measure, by Lemma~\ref{lem:pbox:is:a:possib},}
    &=\unepbox([0_\pspace,x))=\udf(x)
  \end{align*}
  using that $x$ has no immediate predecessor and Eqs.~\eqref{eq:nex-intervals}. The converse inequality follows from the non-decreasingness of
  $\udf$. 

  Next, assume that, ex absurdo,
  $B^c=\{x\in\pspace^*\colon\ldf(x)=1\}$ has no minimum.
  This simply means that for every
  $x\in B^c$ there is an $x'\in B^c$ such that
  $x'\prec x$.
  So, in particular, $\ldf(x)=\ldf(x-)=1$
  for all $x$ in
  $B^c$, and hence,
  \begin{equation*}
    \unepbox(B^c)=\sup_{x\in B^c}\unepbox(\{x\})=\sup_{x\in B^c}(\udf(x)-\ldf(x-))=0.
  \end{equation*}
  Yet, also,
  \begin{equation*}
    \unepbox(B^c)=1
  \end{equation*}
  by Eq.~\eqref{eq:unex-0-1-ldf:cases}.
  We arrived at a contradiction.

  Finally, we show that
  Eq.~\eqref{eq:pbox-iff-possibility-0-1-ldf} holds. By
  Eq.~\eqref{eq:unex-0-1-ldf:min},
  \begin{equation*}
   \unepbox(A)=\min_{y\in B^c}\inf_{x \in \pspace^*\colon A \cap [0_\pspace,y] \preceq x}
   \udf(x)=\inf_{x \in \pspace^*\colon A \cap [0_\pspace,\min B^c] \preceq x}
   \udf(x)=\unepbox(A'),
  \end{equation*}
  with $A':=A\cap [0_\pspace,\min B^c]$. Since $\unepbox$ is a
  possibility measure, we conclude that
  \begin{equation*}
   \unepbox(A)=\unepbox(A')=\sup_{x \in A'} \unepbox(\{x\})=\sup_{x \in A'}
   \udf(x),
  \end{equation*}
  because  $\unepbox(\{x\})=\udf(x)-\ldf(x-)=\udf(x)$, since $\ldf(x-)=0$ for all $x \in [0,\min B^c]$. Hence, Eq.~\eqref{eq:pbox-iff-possibility-0-1-ldf}
  holds.

  ``if''. The claim is established if we can show that
  Eq.~\eqref{eq:pbox-iff-possibility-0-1-ldf} holds, because then
  \begin{equation*}
   \unepbox(A)=\sup_{x\in A\cap[0_\pspace,\min B^c]}\udf(x)=\sup_{x\in A\cap[0_\pspace,\min
   B^c]}\unepbox(\{x\}) \preceq \sup_{x \in A} \unepbox(\{x\}),
  \end{equation*}
  and the converse inequality follows from the monotonicity of
  $\unepbox$.

  Consider any
  event $A\subseteq\pspace$, and let $y$ be a supremum of
  $A'=A\cap[0_\pspace,\min B^c]$ (which exists because
  $\pspace/\simeq$ is order complete), so $\unepbox(A)=\udf(y)$ by
  Eq.~\eqref{eq:unex-0-1-ldf-order-complete-minimum}.
  If $y$ has an
  immediate predecessor, then $A'$ has a maximum (as we will show
  next), and
  \begin{equation*}
    \unepbox(A)=\udf(y)=\udf(\max A')=\max_{x\in A'}\udf(x)=\sup_{x\in A'}\udf(x).
  \end{equation*}
  If $y$ has no immediate predecessor, then either $A'$ has a maximum,
  and the above argument can be recycled,
  or $A'$ has no maximum, in
  which case
  \begin{equation*}
    \unepbox(A)=\udf(y)=\udf(y-)=\sup_{x\in A'}\udf(x).
  \end{equation*}
   The last equality holds because
  \begin{align*}
     \udf(y-)
     &=
     \sup_{x\prec \sup A'}\udf(x)
     \\
     \intertext{and, $A'$ has no maximum, so for every $x\prec \sup A'$, there is an $x'\in A'$ such that $x\prec x'\prec \sup A'$,
     whence}
     &=
     \sup_{x'\in A'}\udf(x').
  \end{align*}

  We are left to prove $A'$ has a
  maximum whenever $y$ has an immediate
  predecessor. Suppose that $A'$ has no maximum. Then it must hold that
  \begin{equation*}
    x\prec y\text{ for all }x\in A'
  \end{equation*}
  since otherwise $x\simeq y$ for some $x\in A'$, whereby $x$ would be
  a maximum of $A'$. 

  But, since $y$ has an immediate predecessor $y-$, the above equation
  implies that
  \begin{equation*}
    x\preceq y-\text{ for all }x\in A'.
  \end{equation*}
  Hence, $y-$ is an upper bound for $A'$, yet $y-\prec y$: this implies
  that $y$ is not a minimal upper bound for $A'$, or in other words,
  that $y$ is not a supremum of $A'$: we arrived at a contradiction.
  We conclude that $A'$ must have a maximum.
\end{proof}

\subsection{P-Boxes with  Zero-One-Valued Upper Cumulative Distribution Functions}

Similarly, we can also determine when
a p-box with  $0$--$1$-valued $\udf$ is a possibility measure:

\begin{proposition}\label{prop:pbox-iff-possibility-0-1-udf}
  Assume that $\pspace/\simeq$ is order complete.
  Let $(\ldf,\udf)$ be a p-box with $0$--$1$-valued $\udf$, and let
  $C=\{x\in\pspace^*\colon\udf(x)=0\}$. Then, $(\ldf,\udf)$ is a possibility
  measure if and only if
  \begin{enumerate}[(i)]
  \item\label{prop:pbox-iff-possibility-0-1-udf:xplus} $\ldf(x)=\ldf(x+)$ for all $x\in\pspace$ that have no
  immediate successor, and
  \item\label{prop:pbox-iff-possibility-0-1-udf:maximum} $C$ has a maximum,
  \end{enumerate}
  and in such a case,
  \begin{equation}\label{eq:pbox-iff-possibility-0-1-udf}
    \unepbox(A)=1-\inf_{y\in A\cap C^c}\ldf(y-).
  \end{equation}
\end{proposition}

Again, for $\unepbox$ to be a possibility measure, both conditions
are still necessary even when $\pspace/\simeq$ is not order
complete: the proof in this direction does not require order
completeness.

As a special case, we mention that $\unepbox$ is a
possibility measure with possibility distribution
\begin{equation*}
  \pi(x)=
  \begin{cases}
    1-\ldf(x-) & \text{if }x\in C^c \\
    0 & \text{otherwise},
  \end{cases}
\end{equation*}
whenever $\pspace/\simeq$ is finite.

\begin{proof}
  ``only if''. Assume that $(\ldf,\udf)$ is a possibility measure.
  For every $x\in\pspace$ that has no immediate successor,
  \begin{align*}
    \ldf(x+)
    &=\inf_{x'\succ x}\ldf(x')
    =\inf_{x'\succ x}\ldf(x'-)
    \\
    \intertext{where the latter equality holds because for every $x'\succ x$ there is an $x''$ such that $x'\succ x''\succ x$; otherwise, $x$ would have an immediate successor.
    Now,
    because $\unepbox(\{x'\})=\udf(x')-\ldf(x'-)$ (see Eq.~\eqref{eq:unex-singleton}),
    $\ldf(x'-)\le 1-\unepbox(\{x'\})$, whence
    }
    &\le\inf_{x'\succ x}(1-\unepbox(\{x'\}))=1-\sup_{x'\succ x}\unepbox(\{x'\})
    \\
    \intertext{and because $(\ldf,\udf)$ is a possibility measure, by Lemma~\ref{lem:pbox:is:a:possib},}
    &=1-\unepbox((x,1_\pspace])=\ldf(x),
  \end{align*}
  where last equality follows from Eq.~\eqref{eq:basic-pbox}. The converse inequality follows from the non-decreasingness of $\ldf$.

  Next, assume that, ex absurdo,
  $C=\{x\in\pspace^*\colon\udf(x)=0\}$ has no maximum.
  Since $\udf(x)=\ldf(x-)=0$
  for all $x$ in
  $C$,
  \begin{equation*}
    \unepbox(C)=\sup_{x\in C}\unepbox(\{x\})=\sup_{x\in C}(\udf(x)-\ldf(x-))=0.
  \end{equation*}
  Yet, also,
  \begin{equation*}
    \unepbox(C)=1
  \end{equation*}
  by Eq.~\eqref{eq:unex-0-1-udf:cases}---indeed, the second case applies because there is no $x\in C$ such that $C\preceq x$, as $C$ has no maximum.
  We arrived at a contradiction.

  Finally, we show that
  Eq.~\eqref{eq:pbox-iff-possibility-0-1-udf} holds. By
  Eq.~\eqref{eq:unex-0-1-udf:max},
  \begin{equation*}
   \unepbox(A)=1-\max_{x\in C}\sup_{y \in \pspace^*\colon y\prec A \cap (x,1_\pspace]}
   \ldf(y)=1-\sup_{y \in \pspace^*\colon y\prec A \cap (\max C,1_\pspace]}
   \ldf(y)=\unepbox(A'),
  \end{equation*}
  with $A':=A\cap (\max C,1_\pspace]=A\cap C^c$. Since $\unepbox$ is a
  possibility measure, we conclude that
  \begin{equation*}
   \unepbox(A)=\unepbox(A')=\sup_{y \in A'} \unepbox(\{y\})=\sup_{y \in A'}
   (1-\ldf(y-))=1-\inf_{y \in A'} \ldf(y-),
  \end{equation*}
  because  $\unepbox(\{y\})=\udf(y)-\ldf(y-)=1-\ldf(y-)$, since $\udf(y)=1$ for all $y \in C^c$.
  Hence, Eq.~\eqref{eq:pbox-iff-possibility-0-1-udf} holds.

  ``if''. The claim is established if we can show that
  Eq.~\eqref{eq:pbox-iff-possibility-0-1-udf} holds, because then
  \begin{equation*}
   \unepbox(A)
   =1-\inf_{y\in A\cap C^c} \ldf(y-)
   =\sup_{y\in A\cap C^c} (1-\ldf(y-))
   \leq \sup_{y \in A} \unepbox(\{y\}),
  \end{equation*}
  and the converse inequality follows from the monotonicity of
  $\unepbox$.

  Consider any
  event $A\subseteq\pspace$, and let $x$ be an infimum of
  $A'=A\cap C^c$ (which exists because
  $\pspace/\simeq$ is order complete).
  If $x$ has an
  immediate successor, then $A'$ has a minimum (as we will show
  next), and
  by Eq.~\eqref{eq:unex-0-1-udf-order-complete-maximum},
  \begin{equation*}
    \unepbox(A)=1-\ldf(\min A'-)=1-\min_{y\in A'}\ldf(y-)=1-\inf_{y\in A'}\ldf(y-).
  \end{equation*}
  If $x$ has no immediate successor, then either $A'$ has a minimum,
  and the above argument can be recycled,
  or $A'$ has no minimum, in
  which case Eq.~\eqref{eq:unex-0-1-udf-order-complete-maximum} implies that
  \begin{equation*}
    \unepbox(A)=1-\ldf(x)=1-\ldf(x+)=1-\inf_{y\in A'}\ldf(y-).
  \end{equation*}
  Here the second equality follows from assumption (i) and the last equality holds because
  \begin{align*}
     \ldf(x+)
     &=
     \inf_{y\succ \inf A'}\ldf(y)
     \\
     \intertext{and, $A'$ has no minimum, so for every $y\succ \inf A'$, there is a $y'\in\pspace$ such that $y\succ y'\succ \inf A'$,
     whence}
     &=
     \inf_{y\succ \inf A'}\sup_{y\succ y'\succ \inf A}\ldf(y')
     =
     \inf_{y\succ \inf A'}\ldf(y-)
     \\
     \intertext{and, again, $A'$ has no minimum, so for every $y\succ \inf A'$, there is a $y''\in A'$ such that $y\succ y''\succ \inf A'$,
     whence}
     &=
     \inf_{y''\in A'}\ldf(y''-).
  \end{align*}

  We are left to prove $A'$ has a
  minimum whenever $x$ has an immediate
  successor. Suppose that $A'$ has no minimum. Then it must hold that
  \begin{equation*}
    y\succ x\text{ for all }y\in A'
  \end{equation*}
  since otherwise $y\simeq x$ for some $y\in A'$, whereby $y$ would be
  a minimum of $A'$. 

  But, since $x$ has an immediate successor $x+$, the above equation
  implies that
  \begin{equation*}
    y\succeq x+\text{ for all }y\in A'.
  \end{equation*}
  Hence, $x+$ is a lower bound for $A'$, yet $x+\succ x$: this implies
  that $x$ is not a maximal lower bound for $A'$, or in other words,
  that $x$ is not an infimum of $A'$: we arrived at a contradiction.
  We conclude that $A'$ must have a minimum.
\end{proof}

\subsection{Necessary and Sufficient Conditions}

Merging
Corollary~\ref{cor:max-suf-nec} with
Propositions~\ref{prop:pbox-iff-possibility-0-1-ldf}
and~\ref{prop:pbox-iff-possibility-0-1-udf} we obtain the following
necessary and sufficient conditions for a p-box to be a
possibility measure:

\begin{corollary}\label{cor:pboxposs-nec-suf}
 Assume that $\pspace/\simeq$ is order complete and let $(\ldf,\udf)$ be a p-box. Then $(\ldf,\udf)$ is a possibility
  measure if and only if either
    \begin{enumerate}
  \item[(L1)] $\ldf$ is $0$--$1$-valued,
  \item[(L2)] $\udf(x)=\udf(x-)$ for all $x\in\pspace$ that have no immediate predecessor, and
  \item[(L3)] $\{x\in\pspace^*\colon\ldf(x)=1\}$ has a minimum,
  \end{enumerate}
  or
  \begin{enumerate}
  \item[(U1)] $\udf$ is $0$--$1$-valued,
  \item[(U2)] $\ldf(x)=\ldf(x+)$ for all $x\in\pspace$ that have no
  immediate successor, and
  \item[(U3)] $\{x\in\pspace^*\colon\udf(x)=0\}$ has a maximum.
  \end{enumerate}
\end{corollary}

This result settles the cases where p-boxes reduce to possibility
measures. We can now go the other way around, and characterise those
cases where possibility measures are p-boxes. Similarly to what
happens in the finite setting, we will see that almost all
possibility measures can be represented by a p-box.

\section{From Possibility Measures to P-Boxes}
\label{sec:poss-to-pbox}

In this section, we discuss and extend some previous results linking possibility distribution to p-boxes. We show that possibility measures correspond to specific kinds of p-boxes, and that some p-boxes correspond to the conjunction of two possibility distribution.

\subsection{Possibility Measures as Specific P-boxes}

\cite{2006:baudrit} already discuss the link between possibility measures and p-boxes defined on the real line with the usual ordering, and they show that any
possibility measure can be approximated by a p-box, however at the
expense of losing some information. We substantially strengthen their result,
and even reverse it: we prove that any possibility measure with compact range can be
\emph{exactly} represented by a p-box with vacuous lower cumulative
distribution function, that is, $\ldf=I_{[1_\pspace]_\simeq}$. In other
words, generally speaking, possibility measures are a special case of
p-boxes on totally preordered spaces.

\begin{theorem}\label{thm:pbox:representation:of:possib}
  For every possibility measure $\Pi$ on $\pspace$ with possibility distribution $\pi$ such that $\pi(\pspace)=\{\pi(x)\colon x\in\pspace\}$ is compact, there is a
  preorder $\preceq$ on $\pspace$ and an upper cumulative distribution
  function $\udf$ such that the p-box
  $(\ldf=I_{[1_\pspace]_\simeq},\udf)$ is a possibility measure with possibility distribution $\pi$, that is, such that for all
  events $A$:
  \begin{equation*}
    \unepbox(A)=\sup_{x\in A}\pi(x).
  \end{equation*}
  In fact, one may take the preorder $\preceq$ to be the one induced by
  $\pi$ (so $x\preceq y$ whenever $\pi(x)\le\pi(y)$) and $\udf=\pi$.
\end{theorem}
\begin{proof}
Let $\preceq$ be the preorder induced by $\pi$.
Order completeness of $\pspace/\simeq$ is satisfied because
$\pi(\pspace)$ is compact with respect to the usual topology on
$\reals$. Indeed, for any $A\subseteq\pspace$, the supremum and
infimum of $\pi$ over $A$ belong to $\pi(\pspace)$ by its compactness,
whence $\pi^{-1}(\inf_{x\in A}\pi(x))$ consists of the infima of $A$,
and $\pi^{-1}(\sup_{x\in A}\pi(x))$ consists of its suprema.

Consider the p-box
$(I_{[1_\pspace]_\simeq},\pi)$. Then, for any $A\subseteq\pspace$, we deduce
from Eq.~\eqref{eq:unex-0-1-ldf-order-complete-vacuous}
that
\begin{equation*}
  \unepbox(A)=\udf(\sup A)=\pi(\sup A)=\sup_{x\in A}\pi(x)
\end{equation*}
because $x\preceq y$ for all $x \in A$
if and only if $\pi(x)\le \pi(y)$ for all $x \in A$, by definition
of $\preceq$, and hence, a minimal upper bound, or supremum, $y$ for
$A$ must be one for which $\pi(y)=\sup_{x\in A}\pi(x)$ (and, again, such $y$
exists because $\pi(\pspace)$ is compact).
\end{proof}

The representing p-box is not necessarily unique:

\begin{example}
Let $\pspace=\{x_1,x_2\}$ and let $\Pi$ be the possibility measure
determined by the possibility distribution
\begin{align*}
\pi(x_1)&=0.5 & \pi(x_2)&=1.
\end{align*}
As proven in
Theorem~\ref{thm:pbox:representation:of:possib}, this possibility
measure can be obtained if we consider the order $x_1 \prec x_2$ and
the p-box $(\ldf_1,\udf_1)$ given by
\begin{align*}
  \ldf_1(x_1)&=0 & \ldf_1(x_2)&=1 \\
  \udf_1(x_1)&=0.5 & \udf_1(x_2)&=1.
\end{align*}
However, we also obtain it if
we consider the order $x_2 \prec x_1$ and the p-box
$(\ldf_2,\udf_2)$ given by
\begin{align*}
  \ldf_2(x_1)&=1 & \ldf_2(x_2)&=0.5 \\
  \udf_2(x_1)&=1 & \udf_2(x_2)&=1.
\end{align*}
The p-box $(\ldf_2,\udf_2)$ induces a possibility measure from
Corollary~\ref{cor:max-suf-nec}, also taking into account that
$\pspace$ is finite. Moreover, by Eq.~\eqref{eq:unex-singleton},
\begin{align*}
 \overline{E}_{\ldf_2,\udf_2}(x_2)
 &=\udf(x_2)-\ldf(x_2-)=1 \\
 \overline{E}_{\ldf_2,\udf_2}(x_1)
 &=\udf(x_1)-\ldf(x_1-)=0.5
\end{align*}
as with the given ordering, $x_2-=0_\pspace-$ and $x_1-=x_2$.
As a consequence, $\overline{E}_{\ldf_2,\udf_2}$ is a possibility
measure associated to the possibility distribution $\pi$.
\end{example}

There are possibility measures which cannot be
represented as p-boxes when $\pi(\pspace)$ is not compact:

\begin{example}
Let $\pspace=[0,1]$, and consider the possibility distribution given
by $\pi(x)=(1+2x)/8$ if $x<0.5$, $\pi(0.5)=0.4$ and $\pi(x)=x$ if
$x> 0.5$; note that $\pi(\pspace)=[0.125,0.25)\cup\{0.4\} \cup
(0.5,1]$ is not compact. The ordering induced by $\pi$ is the usual
ordering on $[0,1]$. Let $\Pi$ be the possibility measure induced by
$\pi$. We show that there is no p-box $(\ldf,\udf)$ on
$([0,1],\preceq)$, regardless of the ordering $\preceq$ on $[0,1]$,
such that $\unepbox=\Pi$.

By Corollary~\ref{cor:max-suf-nec}, if $\unepbox=\Pi$, then at least
one of $\ldf$ or $\udf$ is $0$-$1$--valued. Assume first that $\ldf$
is $0$-$1$--valued. By Eq.~\eqref{eq:unex-singleton},
$\unepbox(\{x\})=\udf(x)-\ldf(x-)=\pi(x)$. Because $\pi(x)>0$ for
all $x$, it must be that $\ldf(x-)=0$ for all $x$, so $\udf=\pi$.
Because $\udf$ is non-decreasing, $x\preceq y$ if and only if
$\udf(x)\le\udf(y)$; in other words, $\preceq$ can only be the usual
ordering on $[0,1]$ for $(\ldf,\udf)$ to be a p-box. Hence,
$\ldf=I_{\{1\}}$.

Now, with $A=[0,0.5)$, we deduce from
Proposition~\ref{prob:unex-0-1-ldf} that
\begin{equation*}
 \unepbox(A)=\inf_{A \preceq x} \udf(x)=0.4 > 0.25=\sup_{x \in A}\pi(x)=\Pi(A),
\end{equation*}
where the second equality follows because $B^c=\{1\}$. Hence, $\unepbox$ does
not coincide with $\Pi$.

Similarly, if $\udf$ would be $0$--$1$-valued, then we deduce from
Eq.~\eqref{eq:unex-singleton} that $\udf(x)=1$ for
every $x$, again because $\pi(x)>0$ for all $x$.
Therefore, $\ldf(x-)=1-\pi(x)$ for all $x$.
But, because $\ldf$ is non-decreasing, $\preceq$ can only be the inverse of the usual ordering on $[0,1]$ for $(\ldf,\udf)$ to be a p-box.

This deserves some explanation. We wish to show that $\ldf(x-)<\ldf(y-)$ implies $x\prec y$. Assume ex absurdo that $x\succ y$. But, then,
\begin{equation*}
  \ldf(x-)=\sup_{z\prec x}\ldf(z)\ge\sup_{z\prec y}\ldf(z)=\ldf(y-),
\end{equation*}
a contradiction. It also cannot hold that $x\simeq y$, because in
that case $z\prec x$ if and only if $z\prec y$, and whence it would
have to hold that $\ldf(x-)=\ldf(y-)$. Concluding, it must hold that
$x\prec y$ whenever $\ldf(x-)<\ldf(y-)$, or in other words, whenever
$x>y$. So, $\preceq$ can only be the inverse of the usual ordering
on $[0,1]$ and, in particular, $[0,1]/\simeq$ is order complete.

Now, for $(\ldf,\udf)$ to induce the possibility measure $\Pi$, we
know from Corollary~\ref{cor:pboxposs-nec-suf} that
$\ldf(x)=\ldf(x+)$ for every $x$ that has no immediate successor in
with respect to $\preceq$, that is, for every $x\prec 0$,
or equivalently, for every $x>0$.
Whence,
\begin{equation*}
  \ldf(x)=\ldf(x+)=\inf_{y \succ x} \ldf(y)=\inf_{y \succ x} \ldf(y-)=1-\sup_{y \succ x}\pi(y)
  =1-\sup_{y<x}\pi(y)
\end{equation*}
for all $x>0$.
This leads to a contradiction: by the definition of $\pi$, we have on the one hand,
\begin{equation*}
  \ldf(0.5-)
  =\sup_{x\prec 0.5} \ldf(x)
  =\sup_{x>0.5} \ldf(x)
  =\sup_{x>0.5} (1-\sup_{y<x}\pi(y))
  =0.5
\end{equation*}
and on the other hand,
\begin{equation*}
  \ldf(0.5-)=1-\pi(0.5)=0.6.
\end{equation*}

Concluding, $\unepbox$ coincides with $\Pi$ in neither case.
\end{example}

Another way of relating possibility measures and p-boxes goes via
random sets (see for instance
\cite{KrieglerHeld2005} and \cite{desterckedubois2008}).
Possibility measures on ordered spaces can also be obtained via
upper probabilities of random sets (see for instance
\cite[Sections~7.5--7.7]{cooman1998} and \cite{miranda2005}).

\subsection{P-boxes as Conjunction of Possibility Measures}

In~\cite{desterckedubois2008}, where p-boxes are studied on finite
spaces, it is shown that a p-box can be interpreted as the
conjunction of two possibility measures, in the sense that
$\solp(\upr_{\ldf,\udf})$ is the intersection of two sets of
additive probabilities induced by two possibility measures. The next
proposition extends this result to arbitrary totally preordered spaces.

\begin{proposition}
Let $(\ldf,\udf)$ be a p-box such that $(\ldf_1=\ldf, \udf_1=I_{\pspace})$ and
$(\ldf_2=I_{[1_\pspace]_\simeq},\udf_2=\udf)$ are possibility measures. Then,
$(\ldf,\udf)$ is the intersection of two possibility measures
defined by the distributions
\begin{align*}
  \pi_1(x)&=1-\ldf(x-) &
  \pi_2(x)&=\udf(x)
\end{align*}
in the sense that $\solp(\upr_{\ldf,\udf})=\solp(\Pi_1) \cap \solp(\Pi_2)$.
\end{proposition}
\begin{proof}
  Using Propositions~\ref{prop:pbox-iff-possibility-0-1-ldf}
  and~\ref{prop:pbox-iff-possibility-0-1-udf}, and the fact that, by
  construction, $[0_\pspace,\min B^c]=C^c=\pspace$, it follows readily
  that $\pi_1$ and $\pi_2$ are the possibility distributions
  corresponding to the p-boxes $(\ldf_1,\udf_1)$ and
  $(\ldf_2,\udf_2)$.

  Thus, by assumption, $\unex_{\ldf_1,\udf_1}=\Pi_1$ and
  $\unex_{\ldf_2,\udf_2}=\Pi_2$. Because natural extensions of two
  coherent lower previsions can only coincide when their credal sets
  are the same
  \cite[\S 3.6.1]{walley1991},
  it follows that
  \begin{align*}
    \solp(\upr_{\ldf_1,\udf_1})&=\solp(\Pi_1)
    &
    \solp(\upr_{\ldf_2,\udf_2})&=\solp(\Pi_2)
  \end{align*}
  We are left to prove that
  \begin{equation*}
    \solp(\upr_{\ldf,\udf})=\solp(\upr_{\ldf_1,\udf_1})\cap\solp(\upr_{\ldf_2,\udf_2})
  \end{equation*}
  but this follows almost trivially after writing down the constraints
  for each p-box.
\end{proof}

This suggests a simple way (already mentioned
in~\cite{desterckedubois2008}) to conservatively approximate
$\lnepbox$ by using the two
possibility distributions:
\begin{equation*}
  \max\{\lnex_{\pi_{\udf}}(A),\lnex_{\pi_{\ldf}}(A)\}
  \leq
  \lnepbox(A)
  \leq
  \unepbox(A)
  \leq
  \min\{\unex_{\pi_{\udf}}(A),\unex_{\pi_{\ldf}}(A)\}.
\end{equation*}
This
approximation is computationally attractive, as it allows us to
use the supremum preserving properties of possibility measures.
However, as next example shows, the approximation will usually be
very conservative, and hence not likely to be helpful.

\begin{example}
Consider $x\prec y \in \pspace$. The distance between
$\unepbox$ and its approximation
$\min\{\unex_{\pi_{\udf}},\unex_{\pi_{\ldf}}\}$ on the interval
$(x,y]$ is given by
\begin{align*}
&\min\{\unex_{\pi_{\udf}}((x,y]),\unex_{\pi_{\ldf}}((x,y])\}-\unex((x,y])
\\
&=\min\{\udf(y),1-\ldf(x)\}-(\udf(y)-\ldf(x))
\\
&=\min\{\ldf(x),1-\udf(y)\}.
\end{align*}
Therefore, the approximation will be close to the exact value on
this set 
only when either $\ldf(x)$ is close to zero or $\udf(y)$ is close to
one.
\end{example}

\section{Natural Extension of $0$--$1$-Valued P-Boxes}
\label{sec:specialcases}

From Proposition~\ref{prob:unex-0-1-ldf}, we can derive an
expression for the natural extension of a $0$--$1$-valued p-box
(see Figure~\ref{fig:01pbox}):

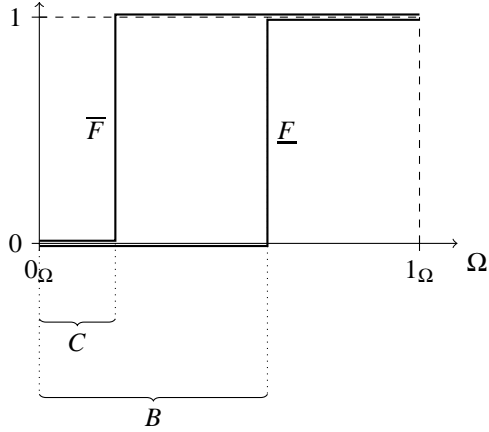
\begin{figure}
\begin{center}
\begin{tikzpicture}
\draw[->] (-0.1,0) node[left] {$0$} -- (5.5,0) node[below right] {$\pspace$};
\draw (5,0.1) -- (5,-0.1) node[below] {$1_\pspace$};
\draw[->] (0,-0.1) node[below] {$0_\pspace$} -- (0,3.2);
\draw (0.1,3) -- (-0.1,3) node[left] {$1$};
\draw[thick,yshift=-1pt] (0,0) -- (3,0) -- (3,3) node[midway,right] {$\ldf$} -- (5,3);
\draw[thick,yshift=+1pt] (0,0) -- (1,0) -- (1,3) node[midway,left] {$\udf$}  -- (5,3) ;
\draw[dashed] (5,0) -- (5,3);
\draw[dashed] (0,3) -- (5,3);
\draw[decorate,decoration={brace,mirror}] (0,-1) -- node[below,yshift=-2pt]{$C$} (1,-1);
\draw[decorate,decoration={brace,mirror}] (0,-2) -- node[below,yshift=-2pt]{$B$} (3,-2) {};
\draw[dotted] (0,0) -- (0,-2);
\draw[dotted] (1,0) -- (1,-1);
\draw[dotted] (3,3) -- (3,-2);
\end{tikzpicture}
\caption{A $0$--$1$-valued p-box.}
\label{fig:01pbox}
\end{center}
\end{figure}

\begin{proposition}\label{prob:unex-0-1-udf-ldf}
Let $(\ldf,\udf)$ be a p-box where $\udf=I_{C^c},\ldf=I_{B^c}$ for
some $C \subseteq B \subseteq \pspace$. Then for any $A \subseteq
\pspace$,
\begin{align}\label{eq:unex-0-1-udf-ldf}
  \unepbox(A)=
  \begin{cases}
    0
    & \text{if there are } x \in C\text{ and }y \in B^c \text{ such that }
    \\
    &A\cap C\preceq x\text{ and }A\cap B\cap C^c=\emptyset\text{ and }y\prec A\cap B^c
    \\
    1
    & \text{otherwise}.
  \end{cases}
\end{align}
\end{proposition}
\begin{proof}
 From Proposition~\ref{prob:unex-0-1-ldf}, the natural extension of
 $(\ldf,\udf)$ is given by
 \begin{equation*}
 \unepbox(A)
  =
  \begin{cases}
    \inf_{x \in \pspace^*\colon A \cap B \preceq x} \udf(x)
    & \text{if }y\prec A\cap B^c\text{ for at least one }y\in B^c,
    \\
    1
    & \text{otherwise}.
  \end{cases}
  \end{equation*}
  Now, if $\udf=I_{C^c}$, the infimum in the first equation is equal to $0$ if and only if there is
  some $x \in C$ such that $A \cap B \preceq x$, and equal to $1$
  otherwise.
  Finally, note that $A \cap B \preceq x$ if and only if
  \begin{align*}
    A\cap B\cap C&\subseteq [0_\pspace,x]\cap C\text{ and }
    \\
    A\cap B\cap C^c&\subseteq [0_\pspace,x]\cap C^c
  \end{align*}
  and observe that $B\cap C=C$, $[0_\pspace,x]\cap C=[0_\pspace,x]$,
  and $[0_\pspace,x]\cap C^c=\emptyset$, to arrive at
  the conditions in Eq.~\eqref{eq:unex-0-1-udf-ldf}.
\end{proof}

Moreover, Propositions~\ref{prop:pbox-iff-possibility-0-1-ldf}
and~\ref{prop:pbox-iff-possibility-0-1-udf} allow us to determine
when this p-box is a possibility measure:

\begin{proposition}\label{prop:pbox-iff-possibility-0-1-ldf-udf}
Assume that $\pspace/\simeq$ is order complete. Let $(\ldf,\udf)$ be
a p-box where $\udf=I_{C^c},\ldf=I_{B^c}$ for some $C \subseteq B
\subseteq \pspace$. Then $(\ldf,\udf)$ is a possibility measure
if and only if $C$ has a maximum and $B^c$ has a minimum.
In such a case,
\begin{equation*}
  \unepbox(A)=
  \begin{cases}
    1 & \text{ if }A\cap (\max C,\min B^c]\neq\emptyset \\
    0 & \text{ otherwise}
  \end{cases}
\end{equation*}
or, in other words, in such a case, $\unepbox$ is a possibility measure with
possibility distribution
\begin{equation*}
  \pi(x)=
  \begin{cases}
    1 & \text{ if }x\in (\max C,\min B^c] \\
    0 & \text{ otherwise}.
  \end{cases}
\end{equation*}
\end{proposition}
\begin{proof}
  ``only if''. Immediate by
  Propositions~\ref{prop:pbox-iff-possibility-0-1-ldf}
  and~\ref{prop:pbox-iff-possibility-0-1-udf}.

  ``if''. By
  Proposition~\ref{prop:pbox-iff-possibility-0-1-ldf},
  $(\ldf,\udf)$ is a possibility measure if and
  only if $B^c$ has a minimum and $\udf(x)=\udf(x-)$ for every $x$
  with no immediate predecessor. The latter condition holds
  for every $x \in C$, and for every $x \in C^c$ with a
  predecessor in $C^c$.
  Whence, we only need to check whether $C^c$ has a
  minimum---if not, then every $x\in C^c$ has a predecessor in
  $C^c$---and if so, that this minimum has an immediate predecessor
  (because obviously $1=\udf(\min C)=\udf(\min C-)=0$ cannot hold).

  Indeed, because $C$ has a maximum, $C^c=(\max C,1_\pspace]$. So,
  either $C^c$ has a minimum, in which case $\max C$ must be
  the immediate predecessor of this minimum, or $C^c$ has no minimum.

  The expression for $\unepbox(A)$ follows from
  Eq.~\eqref{eq:unex-0-1-udf-ldf}. In that equality, without loss of
  generality, we can take $x=\max C$ and $y=\min B^c$, and
  $A\cap C\preceq \max C$ is obviously always satisfied, so
  \begin{equation*}
    \unepbox(A)
    =
    \begin{cases}
      0
      & \text{if } A\cap B\cap C^c=\emptyset\text{ and }\min B^c \prec A\cap B^c
      \\
      1
      & \text{otherwise}
    \end{cases}
  \end{equation*}
  So, to establish the desired equality, it suffices to show that
  $A\cap B\cap C^c=\emptyset$ and $\min B^c \prec A\cap B^c$ if and
  only if $A\cap(\max C,\min B^c]=\emptyset$.

  Indeed, $\min B^c\prec A\cap B^c$ precisely when $\min B^c\not\in A$.
  Moreover,
  \begin{equation*}
    B\cap C^c=[0_\pspace,\min B^c)\cap (\max C,1_\pspace]=(\max C,\min B^c).
  \end{equation*}
  So, the desired equivalence is established.
\end{proof}

In particular, we can characterise under which conditions a precise
p-box, i.e., one where $\ldf=\udf:=\df$, induces a possibility
measure. The natural extension of precise p-boxes on the unit
interval was considered in \cite[Section~3.1]{miranda2006c}. From
Proposition~\ref{prop:pbox-possibility-zero-one}, the natural
extension of $\df$ can only be a possibility measure when $\df$ is
0--1-valued. If we apply Proposition~\ref{prob:unex-0-1-udf-ldf}
with $B=C$ we obtain the following:

\begin{corollary}\label{co:nex-precise-01}
Let $(\ldf,\udf)$ be a precise p-box where $\udf=\ldf$ is
$0$--$1$-valued, and let $B=\{x\in\pspace^*\colon\ldf(x)=0\}$. Then,
for every subset $A$ of $\pspace$,
\begin{equation*}
 \unepbox(A)=\begin{cases}
  0 &\text{ if there are } x \in B, y \in B^c\text{ such that }A\cap B\preceq x\text{ and }y\prec A\cap B^c \\
  1 &\text{ otherwise}.
  \end{cases}
\end{equation*}
\end{corollary}
\begin{proof}
  Immediate from Proposition~\ref{prob:unex-0-1-udf-ldf}.
\end{proof}

Moreover, Proposition~\ref{prop:pbox-iff-possibility-0-1-ldf-udf}
allows us to determine when this p-box is a possibility measure:

\begin{corollary}\label{co:precise-01-possib}
Assume that $\pspace/\simeq$ is order complete. Let $(\ldf,\udf)$ be
a precise p-box where $\udf=\ldf$ is $0$--$1$-valued, and let
$B=\{x\in\pspace^*\colon\ldf(x)=0\}$. Then, $(\ldf,\udf)$ is a
possibility measure if and only if $B$ has a maximum and $B^c$ has a
minimum. In that case, for every $A\subseteq\pspace$,
\begin{equation*}
 \unepbox(A)=\begin{cases}
  0 &\text{ if } \min B^c \notin A \\
  1 &\text{ otherwise}.
  \end{cases}
\end{equation*}
\end{corollary}
\begin{proof}
Immediate by Proposition~\ref{prop:pbox-iff-possibility-0-1-ldf-udf}
and Corollary~\ref{co:nex-precise-01}.
\end{proof}

As a consequence, we deduce that a precise $0$--$1$-valued p-box
on $(\pspace,\preceq)=([0,1],\leq)$ never induces a possibility measure---except when
$\df=I_{[0,1]}$. Indeed, if $\df\neq I_{[0,1]}$, then $B\cap [0,1]\neq \emptyset$,
and the maximum of $B$ would need to have an immediate successor
(the minimum of $B^c$), which cannot be for the usual ordering
$\leq$.

When $\df=I_{[0,1]}$, we obtain $\max B=0-$ and $\min
B^c=0$, whence applying Corollary~\ref{co:precise-01-possib} we
deduce that $(\df,\df)$ is a possibility measure, with
possibility distribution
\begin{equation*}
 \pi(x)=\begin{cases}
  1 &\text{ if } x=0 \\
  0 &\text{ otherwise}.
  \end{cases}
\end{equation*}

To see why the possibility distribution
\begin{equation*}
 \pi(x)=\begin{cases}
  1 &\text{ if } x=y \\
  0 &\text{ otherwise}
  \end{cases}
\end{equation*}
for any $y>0$ does not correspond to the precise p-box $(\df,\df)$ with $\df=I_{[y,1]}$, first note that
\begin{equation*}
  \Pi([0,y))=\sup_{x\prec y} \pi(x)=0.
\end{equation*}
But, for $\Pi$ to be the p-box $\unex_{\df,\df}$, we also require that
\begin{equation*}
  \unex_{\df,\df}([0,y))=\df(y)-\df(0-)=1
\end{equation*}
using Eq.~\eqref{eq:nex-intervals}, because $y$ has no
immediate predecessor. Whence, we arrive at a contradiction.

\section{Constructing Multivariate Possibility Measures from Marginals}
\label{sec:marginals}

In \cite{unpub:troffaesdestercke::pboxes:multivariate}, multivariate
p-boxes were constructed from marginals. We next apply this
construction together with the p-box representation of possibility
measures, given by Theorem~\ref{thm:pbox:representation:of:possib},
to build a joint
possibility measure from some given marginals. As particular cases,
we consider the joint,
\begin{enumerate}[(i)]
\item either without any assumptions about
dependence or independence between variables, that is, using the
Fr\'echet-Hoeffding bounds \cite{hoeffding1963},
\item or assuming epistemic independence between all variables, which allows us to use the factorization
property \cite{2010:decooman:ine}.
\end{enumerate}

Let us consider $n$ variables $X_1$, \dots, $X_n$ assuming values in $\values_1$, \dots, $\values_n$. Assume that
for each variable $X_i$ we are given a possibility measure $\Pi_i$
with corresponding possibility distribution $\pi_i$ on
$\values_i$. We assume that the range of all marginal
possibility distributions is $[0,1]$; in particular,
Theorem~\ref{thm:pbox:representation:of:possib} applies, and each
marginal can be represented by a p-box on $(\values_i,\preceq_i)$,
with vacuous $\ldf_i$, and $\udf_i=\pi_i$. Remember that the
preorder $\preceq_i$ is the one induced by $\pi_i$.

\subsection{Multivariate Possibility Measures}

The construction in
\cite{unpub:troffaesdestercke::pboxes:multivariate} employs the
following mapping $Z$, which induces a preorder $\preceq$ on
$\pspace=\values_1\times\dots\times \values_n$:
\begin{equation}\label{eq:joint:z}
  Z(x_1,\dots,x_n)=\max_{i=1}^n \pi_i(x_i).
\end{equation}
With this choice of $Z$, we can easily
find the possibility measure which represents the joint as accurately as possible, under any rule of combination of coherent lower probabilities:

\begin{lemma}\label{lem:joint-p-box}
  Let $\odot$ be any rule of combination of coherent upper probabilities, mapping the marginals $\upr_1$, \dots, $\upr_n$ to a joint coherent upper probability $\bigodot_{i=1}^n\upr_i$ on all events. If there is a continuous function $u$ for which
  \begin{align*}
    \bigodot_{i=1}^n\upr_i\left(\prod_{i=1}^n A_i\right)
    &=
    u(\upr_1(A_1),\dots,\upr_n(A_n))
  \end{align*}
  for all $A_1\subseteq \values_1$, \dots, $A_n\subseteq \values_n$, then the
  possibility distribution $\pi$ defined by
  \begin{align*}
    \pi(x)&=u(Z(x),\dots,Z(x))
  \end{align*}
  induces the least conservative upper cumulative distribution
  function on $(\pspace,\preceq)$ that dominates the combination
  $\bigodot_{i=1}^n\Pi_i$ of $\Pi_1$, \dots, $\Pi_n$.
\end{lemma}
\begin{proof}
  To apply
  \cite[Lem.~22]{unpub:troffaesdestercke::pboxes:multivariate}, we
  first must consider the upper cumulative distribution functions, which in our case coincide with the possibility distributions, as functions on
  the unit interval $z\in[0,1]$. For the marginal possibility
  distribution $\pi_i$, the preorder is the one induced by $\pi_i$
  itself, so, as a function of $z$, $\pi_i$ is simply the identity
  map:
  \begin{align*}
    \pi_i(z)=\pi_i(\pi_i^{-1}(z))=z.
  \end{align*}
  Using \cite[Lem.~22]{unpub:troffaesdestercke::pboxes:multivariate}, the
  least conservative upper cumulative distribution function on the space
  $(\pspace,\preceq)$ that dominates the combination
  $\bigodot_{i=1}^n\Pi_i$ is given by
  \begin{equation*}
    \udf(z)=u(\pi_1(z),\dots,\pi_n(z))=u(z,\dots,z) \ \forall z \in [0,1].
  \end{equation*}
  As a function of $x\in\pspace$, this means that
  \begin{equation*}
    \udf(x)=u(Z(x),\dots,Z(x))
  \end{equation*}
  with the $Z$ that induced $\preceq$, that is, the one defined in
  Eq.~\eqref{eq:joint:z}.

  Now, by Proposition~\ref{prop:pbox-iff-possibility-0-1-ldf}, such upper cumulative distribution function corresponds to a
  possibility measure with possibility distribution
  \begin{equation*}
    \pi(x)=u(Z(x),\dots,Z(x))
  \end{equation*}
  whenever $\udf(x)=\udf(x-)$ for all $x\in\pspace$ that have no
  immediate predecessor, that is, whenever
  \begin{equation*}
    \udf(x)=\sup_{y\colon Z(y)<Z(x)}\udf(y)
  \end{equation*}
  for all $x$ such that $Z(x)>0$.
  But this must hold, because (i) the range of $Z$ is $[0,1]$, so
  $Z(y)$ can get arbitrarily close to $Z(x)$ from below, and (ii) $u$
  is continuous, so $\udf(y)=u(Z(y),\dots,Z(y))$ gets arbitrarily
  close to $\udf(x)=u(Z(x),\dots,Z(x))$.
\end{proof}

\subsection{Natural Extension: The Fr\'echet Case}

The \emph{natural extension} $\boxtimes_{i=1}^n\upr_i$ of $\upr_1$,
\dots, $\upr_n$ is the upper envelope of all joint (finitely
additive) probability measures whose marginal distributions are
compatible with the given marginal upper probabilities. So, the
model is completely vacuous (that is, it makes no assumptions) about
the dependence structure, as it includes all possible forms of
dependence. See \cite[p.~120, \S 3.1]{cooman2003c} for a rigorous
definition. In this paper, we only need the following equality,
which is one of the Fr\'echet bounds (see for instance \cite[p.~122,
\S 3.1.1]{walley1991}):
\begin{align}
  \label{eq:joint-natural-extension}
  \bigboxtimes_{i=1}^n\upr_i\left(\prod_{i=1}^n A_i\right)
  &=\min_{i=1}^n\upr_i(A_i)
\end{align}
for all $A_1\subseteq \values_1$, \dots, $A_n\subseteq \values_n$.

\begin{theorem}\label{th:joint-alldep}
  The possibility distribution
  \begin{align*}
    \pi(x)&=\max_{i=1}^n\pi_i(x_i)
  \end{align*}
  induces the least conservative upper cumulative distribution
  function on $(\pspace,\preceq)$ that dominates the natural extension
  $\boxtimes_{i=1}^n\Pi_i$ of $\Pi_1$, \dots, $\Pi_n$.
\end{theorem}
\begin{proof}
  Immediate, by Lemma~\ref{lem:joint-p-box} and
  Eq.~\eqref{eq:joint-natural-extension}.
\end{proof}

\subsection{Independent Natural Extension}

In contrast, we can consider joint models which satisfy the property
of \emph{epistemic independence} between the different $X_1$, \dots,
$X_n$. These have been studied in \cite{miranda2003} in the case of
two marginal possibility measures. The most conservative of these
models is called the \emph{independent natural extension}
$\otimes_{i=1}^n\lpr_i$ of $\lpr_1$, \dots, $\lpr_n$. See
\cite{2010:decooman:ine} for a rigorous definition and properties,
and \cite{miranda2003} for a study of joint possibility measures
that satisfy epistemic independence in the case of two variables. In
this paper, we only need the following equality for the independent
natural extension:
\begin{align}
  \label{eq:joint-independent-natural-extension}
  \bigotimes_{i=1}^n\upr_i\left(\prod_{i=1}^n A_i\right)
  &=
  \prod_{i=1}^n\upr_i(A_i)
\end{align}
for all $A_1\subseteq \values_1$, \dots, $A_n\subseteq \values_n$.

\begin{theorem}\label{th:joint-independent}
  The possibility distribution
  \begin{align}
    \label{eq:joint-independent}
    \pi(x)&=\left(\max_{i=1}^n\pi_i(x_i)\right)^n
  \end{align}
  induces the least conservative upper cumulative distribution
  function on $(\pspace,\preceq)$ that dominates
  the independent natural extension
  $\otimes_{i=1}^n\Pi_i$ of $\Pi_1$, \dots, $\Pi_n$.
\end{theorem}
\begin{proof}
  Immediate, by Lemma~\ref{lem:joint-p-box} and
  Eq.~\eqref{eq:joint-independent-natural-extension}.
\end{proof}
Note, however, that there is no least conservative possibility
measure that corresponds to the independent natural extension of
possibility measures \cite[Sec.~6]{miranda2003}.

We do not consider the minimum rule and the product
rule
\begin{equation*}
  \min_{i=1}^n\pi_i(x_i)\text{ and }\prod_{i=1}^n\pi_i(x_i),
\end{equation*}
as their relation with the theory of coherent lower previsions is
still unclear. However, we can compare the above approximation with
the following outer approximation given by
\cite[Proposition~1]{2009:destercke:consonant}:
\begin{equation}
  \label{eq:rsi:possibility}
  \pi(x)=\min_{i=1}^n(1-(1-\pi_i(x_i))^n).
\end{equation}
The above equation is an outer approximation in case of \emph{random
set independence}, which is slightly more
conservative than the independent natural extension
\cite[Sec.~4]{couso1999b}, so in particular, it is also an outer approximation
of the independent natural extension.
Essentially, each distribution $\pi_i$ is transformed into
$1-(1-\pi_i)^n$ before applying the minimum rule. It can be
expressed more simply as
\begin{equation*}
  1-\max_{i=1}^n(1-\pi_i(x_i))^n.
\end{equation*}
If for instance $\pi(x_i)=1$ for at least one $i$, then this formula
provides a more informative (i.e., lower) upper bound than
Theorem~\ref{th:joint-independent}. On the other hand, when all
$\pi(x_i)$ are, say, less than $1/2$, then
Theorem~\ref{th:joint-independent} does better.

Finally, note that neither Eq.~\eqref{eq:joint-independent} nor
Eq.~\eqref{eq:rsi:possibility} are proper joints, in the sense that,
in both cases, the marginals of the joint are outer approximations of
the original marginals, and will in general not coincide with the
original marginals.

\section{Conclusions}\label{sec:conclusions}

Both possibility measures and p-boxes can be seen as coherent upper
probabilities. We used this framework to study the relationship
between possibility measures and p-boxes. Following
\cite{unpub:troffaesdestercke::pboxes:multivariate}, we allowed
p-boxes on arbitrary totally preordered spaces, whence including
p-boxes on finite spaces, on real intervals, and even multivariate
ones.

We began by considering the more general case of maxitive measures,
and proved that a necessary and sufficient condition for a p-box to
be maxitive is that at least one of the cumulative distribution
functions of the p-box must be $0$--$1$ valued. Moreover, we
determined the natural extension of a p-box in those cases and gave
a necessary and sufficient condition for the p-box to be
supremum-preserving, i.e., a possibility measure. As special cases,
we also studied degenerate p-boxes, and precise $0$--$1$ valued
p-boxes.

Secondly, we showed that almost every possibility measure can be
represented as a p-box. Hence, in general, p-boxes are more
expressive than possibility measures, while still keeping a
relatively simple representation and calculus
\cite{unpub:troffaesdestercke::pboxes:multivariate}, unlike many
other models, such as for instance lower previsions and credal sets,
which typically require far more advanced techniques, such as linear
programming.

Finally, we considered the multivariate case in more detail, by
deriving a joint possibility measure from given marginals using the
p-box representation established in this paper and results from
\cite{unpub:troffaesdestercke::pboxes:multivariate}.

In conclusion, we established new connections between both models,
strengthening known results from literature, and allowing many
results from possibility theory to be embedded into the theory of
p-boxes, and vice versa.

As future lines of research, we point out the generalisation of a
number of properties of possibility measures to p-boxes, such as the
connection with fuzzy logic \cite{zadeh1978a} or the representation
by means of graphical structures \cite{borgelt2000}, and the study
of the connection of p-boxes with other uncertainty models, such as
clouds and random sets.

\section*{Acknowledgements}

Work partially supported by projects TIN2008-06796-C04-01 and
MTM2010-17844 and by a doctoral grant from the IRSN. We are also
especially grateful to Gert de Cooman and Didier Dubois for the many
fruitful discussions, suggestions, and for helping us with an
earlier draft of this paper.

\bibliographystyle{plain}
\bibliography{general,transitional}
\end{document}